\documentclass{article} 

\usepackage{amsthm,amsmath,amsfonts,amssymb,hyperref}
\usepackage{graphicx,amsmath,xcolor,wrapfig}
\usepackage{subcaption}
\usepackage{verbatim} 
\usepackage{enumerate} 
\usepackage{stmaryrd} 

\def\Riem{\mathop{\rm Rm}}
\def\Ric{\mathop{\rm Ric}}
\def\cRic{\mathop{\rm R\i\makebox[0pt]{\raisebox{5pt}{\tiny$\circ$\;\,}}c}}

\def\XXint#1#2#3{{\setbox0=\hbox{$#1{#2#3}{\int}$ }
\vcenter{\hbox{$#2#3$ }}\kern-.5\wd0}}

\newtheorem{theorem}{Theorem}[section]
\newtheorem{lemma}[theorem]{Lemma}
\newtheorem{proposition}[theorem]{Proposition}

\theoremstyle{definition}

\theoremstyle{remark}

\title{A Gap Theorem for Half-Conformally Flat Manifolds}
\author{Martin Citoler-Saumell and Brian Weber}
\date{July 21, 2019}

\begin{document}

\maketitle

\section{Introduction}

Any half-conformally flat 4-manifold has a metric of constant scalar curvature in its conformal class.
Such manifolds are divided into three classes; such a manifold is of {\it negative}, {\it zero}, or {\it positive} type depending on the valus of this constant.
A classificaiton of all compact half-conformally flat metrics, even up to conformal equivalence, is currently far away, although on compact manifolds of positive type neccesarily the intersection form is negative definite, and in the zero scalar curvature case the intersection form may have signature $(k,n)$ only for $k=0$, $1$, or $3$ \cite{LeB}.
However little else can be said unless other restrictions are in place such as simple connectedness.
The main result of this paper is a gap theorem, giving restrictions on the betti numbers when scalar curvature is negative but small.

This result is a weakened extension of the LeBrun result into half-conformally flat manifolds of negative type, and it is important to recall LeBrun's method.
Following \cite{LeB}, any harmonic representative $\eta$ of $H^{2}(M;\mathbb{R})$ that has $*\eta=+\eta$ and $\triangle\eta=\frac{s}{3}\eta$.
In case $s>0$ the maximum principle makes solutions impossible, and when $s=0$ then any solution is covariant-constant.
The $b^+=1$ case consists precisely of the compact scalar-flat K\"ahler manifolds that are not Ricci-flat.
When $b^+=3$ the manifold is hyper-K\"ahler, and fully classified: such a manifold must be a quotient of a $K3$ surface or a torus.
In the zero case with $b^+=0$, little is known unless the manifold is simply-connected, in which case $M^4$ is the connected sum $\sharp_1^k\overline{\mathbb{C}P}{}^2$ where $k\ge6$ \cite{L4}.
There is no $b^+=2$ case since any two K\"ahler forms on a 4-manifold automatically create a third: their mutual perpendicular in the $\bigwedge{}^+$ bundle.
Likewise $b^+>3$ is impossible, as $dim\left(\bigwedge{}^+\right)=3$ so $\bigwedge{}^+$ has at most three independent covariant-constant sections.

Even less has is known about negative-type half-conformally flat manifolds, aside from the fact that they are plentiful \cite{T}.
This paper adds to our fundamental knowledge of negative-type manifolds.
Ordinarily $L^2$ control over curvature cannot control topology, but our Theorem \ref{ThmMain} says that in the non-collapsed setting, if scalar curvature $s$ is negative but sufficiently close to 0 then its topology is controlled.
\begin{theorem} \label{ThmMain}
	Assume $(M^4,g)$ is compact and anti-self dual of negative type, with constant scalar curvature $s<0$.
	Assume the Sobolev constant $C_S$ is bounded on definite size balls:
	\begin{eqnarray}
		\left(\int_{B_r(p)}\varphi^4\right)^{\frac14}
		\;\le\;C_S\,\left(\int_{B_r(p)}|\nabla\varphi|^2\right)^{\frac12} \label{IneqSobolev}
	\end{eqnarray}
	whenever $Vol\left(B_r(p)\right)\le\frac1N{}Vol{}M^4$, and $\varphi\in{}C_c^{0,1}(B_r(p))$
	Assume there exist a constant $\Lambda<\infty$ so that we have the topological bound $\chi+3\tau\,>\,-\Lambda^2$.
	
	Then there exists some $\epsilon=\epsilon(\Lambda,C_S)$ so that $s^2Vol(M)<\epsilon$ implies the three betti numbers $b_1$, $b^+$, $b^-$ are uniformly bounded:
	\begin{eqnarray}
		b_1,\,b^+,\,b^-\;\le\;C
	\end{eqnarray}
	where $C$ depends only on $\Lambda$.
\end{theorem}
{\bf Remark}.
In a celebrated work \cite{Gro}, Gromov bounded all betti numbers on compact Riemannian manifolds from pointwise bounds on the curvature tensor and a diameter bound.
Certain weakenings of this result are now well known, but they normally require $L^p$ bounds on curvature for some large $p$, as in \cite{Gao2} (which also requires a Sobolev constant), \cite{Y}, or \cite{PSW}, or still require some pointwise bound on at least Ricci curvature as in \cite{Gao1}.

{\bf Remark}. The proof produces a control on $b^+$ by partially imitating the argument from \cite{LeB}.
In that work $s\ge0$, while in our case we choose sequences of metrics so $s<0$ but $s\nearrow0$.
The inequalities $\chi+3\tau>-\Lambda^2$ and $s^2Vol<\epsilon$ automatically control both $\chi$ and $\tau$.
After $b^+$ is controlled, which we acheive below through analytic means, then all three betti numbers $b_1$, $b^+$, $b^-$ are controlled.

{\bf Remark}.
Given $s^2Vol(M)<\epsilon$, then the topological condition $\chi+3\tau>-\Lambda^2$ is equivalent to an $L^2$ condition on curvature: with $W^+=0$ the usual Chern-Gauss-Bonnet formulas give
\begin{eqnarray}
	\begin{aligned}
		-8\pi^2\left(\chi+3\tau\right)
		&\;=\;-\frac{1}{12}s^2Vol(M)\,+\,\int\frac{s^2}{24}\,+\,\frac12|\cRic|^2\,+\,\frac14|W^-|^2 \\
		&\;=\;-\frac{1}{12}s^2Vol(M)\,+\,\frac14\int|\Riem|^2.
	\end{aligned} \label{EqnRmBounds}
\end{eqnarray}
Control over $L^2(|\Riem|)$ gives control over $\chi$ and $\tau$, as $8\pi^2\chi=\int\frac14|\Riem|^2-|\cRic|^2$ and $-12\pi^2\tau=\int\frac14|W^-|^2$.

{\bf Remark}.
It is notable that the constant $C=C(\Lambda)$ does not depend on $C_S$.
The Sobolev constant is required in order to uniformly bound any closed section $\omega\in\bigwedge^+$, which has th effect of prohibiting $L^2(|\omega|)$ from accumulating inside the bubbbles---the bound on $C_S$ is completely immaterial, just as long as there is {\it some} bound.
We conjecture that Theorem \ref{ThmMain} is true without the assumption on the Sobolev constant.

\vspace{0.15in}

In Section \ref{SecExamples} our paper gives examples of various models for potential bubbles.
These models demonstrate that the standard elliptic techniques cannot produce results any better than those of Theorem \ref{ThmMain}.
They show that, even though $|\omega|$ can be uniformly bounded when $\omega\in\bigwedge^+$, $d\omega=0$, it is {\it im}possible that $|\nabla\omega|$ can be bounded.
Since a scalar flat limiting multifold might look like the one-pont union of several compact orbifolds, the behavior of $\omega$ vary from one component orbifold to the next by undergoing large changes within the bubbles themselves.

For instance consider the situation of Figure 1.
\begin{figure}
	\centering
	\begin{subfigure}[b]{0.45\textwidth}
		\includegraphics[scale=0.45,viewport=0 100 370 260,clip=true]{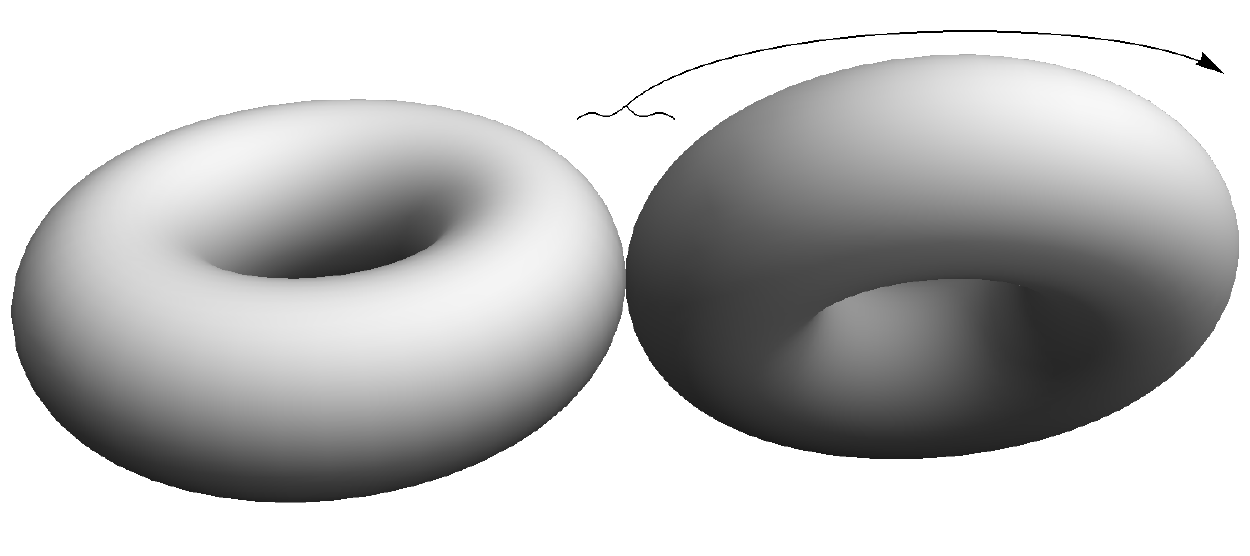}
		\caption{Two larger manifolds connected by a 2-ended bubble, converging to a one-point union.
			Closed form $\omega\in\bigwedge^+$ becomes covariant-constant on the components, but perhaps not in the bubble.}
	\end{subfigure}
	\hspace{0.2in}
	\begin{subfigure}[b]{0.45\textwidth}
		\reflectbox{\includegraphics[width=150px,height=75px,viewport=0 125 500 450,clip=true]{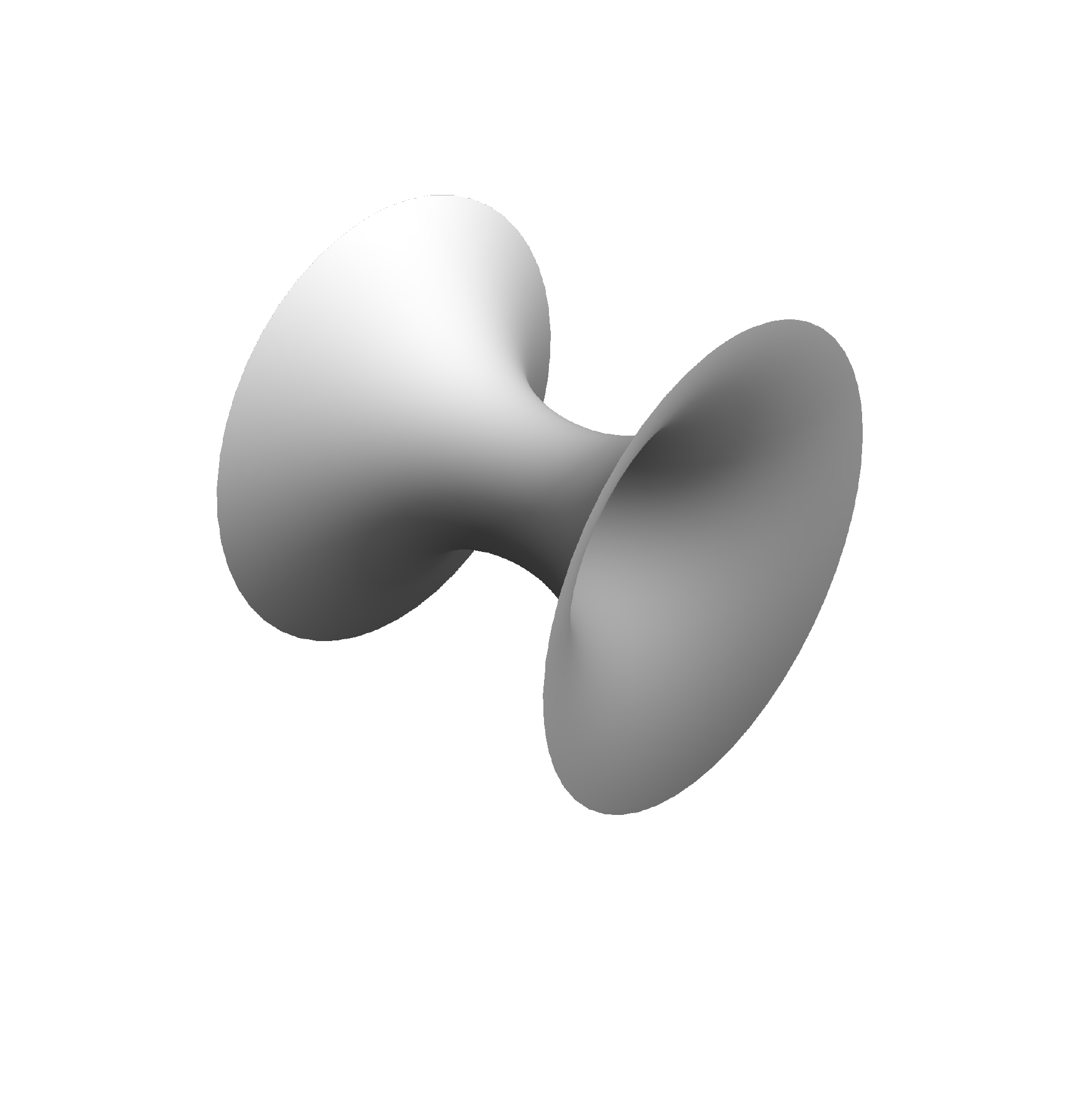}}
		\caption{Two-ended singularity model.
			Scalar-flat, half-conformally flat, ALE manifold with bounded, asymptotically K\"ahler but non-constant solution $\omega\in\bigwedge^+$, $d\omega=0$.
			\textcolor{white}{hi \\}}
	\end{subfigure}
	\caption{Possible degeneration as $s\nearrow0$.}
\end{figure}
Each component of Figure 1(a), in the limit, is scalar-flat and anti-self dual, and so $\omega$ is covariant-constant on each component.
But there can be a ``switching'' behavior within the bubble itself, represented in Figure 1(b), where $\omega$ rapidly changes from one covariant-constant form of behavior to a different covariant-constant form of behavior.
The form $\omega$ might even switch from a covariant-constant, non-zero form on one component to the zero form on the other.

We close the paper with Section \ref{SecExamples} where we show this kind of switching behavior cannot be ruled out.
In Section \ref{SubSec2EndedAK} we build examples of scalar-flat, half-conformally flat 2-ended ALE manifolds (as in Figure 1(b)) with a closed, bounded section $\omega\in\bigwedge^+$ that has different asymptotic behavior on the two ends of the manifold.
To do so, we must solve $d\omega=0$, which is a first order overdetermined system, with certain boundary conditions.
Solving overdetermined first order systems with boundary conditions is usually not easy, but we are able to do this under certain limited conditions by employing a separation of variables method in Section \ref{SubSecSeparation}.
After separating variables, we obtain the following overdetermined evolution equation
\begin{eqnarray}
	div(\eta)\;=\;0, \quad\quad \dot\eta\;=\;curl(\eta) \label{EqnEuclMax}
\end{eqnarray}
where $\eta\in\bigwedge^1_{\mathbb{S}^3}$ is a time-varying field of 1-forms on $\mathbb{S}^3$.
These are the so-called {\it Euclidean-Maxwell} equations, named such because, setting $\eta=B+\sqrt{-1}E$ where $B$ is the magnetic field and $E$ is the electric field, the source-free Maxwell equations on $\mathbb{R}^{1,3}$ are preciesly $div(\eta)=0$, $\dot\eta=\sqrt{-1}curl(\eta)$, so equations (\ref{EqnEuclMax}) are just the Wick rotations of the standard Maxwell equations.

In Theorem \ref{PropDecayGapOfForms} we find that solutions $\omega\in\bigwedge^+$, $d\omega=0$ on ALE manifolds display a distinct gap in their possible asymptotic behaviors.
If the norm $|\omega|$ is bounded on an ALE end, then its asymptotic behavior falls strictly into one of two types: it is either {\it asymptotically K\"ahler}, which means that upon taking a blowdown of the end the form converges to a K\"ahler form, or else its decay rate is fast: $|\omega|=O(r^{-4})$.
Decay rates between $O(1)$ and $O(r^{-4})$ are forbidden, even though all other integer orders of decay $O(r^{-5})$, $O(r^{-6})$, etc, may occur.
The fundamntal reason for this can be traced to the fact that the operator $*d:\bigwedge^1_{\mathbb{S}^3}\rightarrow\bigwedge^1_{\mathbb{S}^3}$, when restricted to divergence-free forms, has a spectral gap: $spec(*d)=\mathbb{Z}\setminus\{-1,0,1\}$.
\begin{theorem}[cf. Theorem \ref{PropDecayGapOfForms}]
	Assume $\omega$ solves $\omega\in\bigwedge^+$, $d\omega=0$ on an ALE manifold end $\Omega$ of dimension 4, and that $\Omega$ is ALE of order at least 2.
	Then either $\omega$ is asymptotically K\"ahler, or else $|\omega|\in{}O(r^{-4})$.
\end{theorem}

{\bf Remark}.
Our examples show that although we can bound $|\omega|$ inside the bubbles, it is quite impossible to bound $|\nabla\omega|$ inside the bubbles.
At the end of Section \ref{SubSec2EndedAK} we produce a 2-ended ALE manifold with a closed form $\omega\in\bigwedge^+$, with $\sup|\omega|=1$ and $L^2(|\nabla\omega|)$ as small as desired, but simultaneaously $\sup|\nabla\omega|$ is as large as desired.

{\bf Remark}.
The Euclidean-Maxwell equations have seen sparse study in the physics literature, in \cite{Schw} \cite{Brill} \cite{Her2} and whimsical treatments in \cite{Zamp2} and \cite{Her1}.
They have seen {\it very} sparse study in the mathematics literature: the only specific mention we could locate in the mathematics literature is a passing note (Example 4.1) in the lecture notes \cite{Bry}.
The physical motivation appears to be the development of a Euclidean theory of EM fields, in the hope that a treatment in Euclidean space-time would yield a more rigorous convergence theory, with the lessons learned potentially carrying over to the Lorenzian world---a hope, it appears, that was never fully realized, largely due to the fact that stable modes are rare, and also to the lack of conservation phenomena.
The papers \cite{Brill} \cite{Her2} are theoretically underdeveloped---even though the older paper of Schwinger's \cite{Schw} is more impressive.
They treat the Euclidean-Maxwell in a naive vector-calculus viewpoint and even after repeating Minkowski's trick \cite{Mink} of placing the Euclidean components of $E$, $B$ into a Maxwell 2-form $F$, it is never noticed that the Euclidean-Maxwell equations decouple.
Indeed the free-space Maxwell equations $dF=0$ and $d*F=0$ split into decoupled equations $dF^+=0$, $dF^-=0$ in the Euclidean but not in the Lorenzian case, as a consequence of the fact that the real vector space $\bigwedge^2_{\mathbb{R}^4}$ splits under the Euclidean but not the Lorenzian Hodge-$*$ operator.

\section{Manifold convergence and linear analysis}

In this section we prove Theorem \ref{ThmMain}.
A crucial factor is the Bochner formula for 2-forms on a 4-manifold:
\begin{eqnarray}
	\triangle\omega
	\;=\;-\triangle_d\omega
	\,-\,W(\omega)\,+\,\frac{s}{3}\omega \label{EqnMainBochner}
\end{eqnarray}
where $\triangle=+div\,grad$ is the rough Laplacian and $\triangle_d=(d^*+d)^2$ is the Hodge Laplacian.
If $\omega\in\bigwedge^+$ is an harmonic representative and $W^+=0$, then (\ref{EqnMainBochner}) is $\triangle\omega=\frac{s}{3}\omega$.
For a scalar version of this equality, we use the classic Kato inequality $\triangle|\omega|\ge\left<\triangle\omega,\omega/|\omega|\right>$ to obtain
\begin{eqnarray}
	\triangle|\omega|\;\ge\;\frac{s}{3}|\omega|
\end{eqnarray}
which holds in the pointwise sense when $\omega$ is non-zero and holds everywhere in, for instance, the distributional or the viscosity sense.

But we can do better: we have an improved Kato inequality for closed sections of $\bigwedge^+$.
If $\omega$ is such a section, then
\begin{eqnarray}
|\nabla|\omega||^2\;\le\;\frac23|\nabla\omega|^2. \label{IneqImprovedKato}
\end{eqnarray}
This inequality first appeared in \cite{Sea}, and also follows from the more general work in \cite{Bra} and \cite{CGH}.
This improved Kato inequality allows for an improved elliptic inequality, and then a better regularity theorem, Proposition \ref{PropRegularityForEta} below.
\begin{lemma}[Improved elliptic inequality] \label{LemmaImprovedElliptic}
	If $\omega\in\bigwedge^+$ is closed, then
	\begin{eqnarray}
		\triangle|\omega|^{\frac12}\ge\frac{s}{6}|\omega|^{\frac12}.
	\end{eqnarray}
\end{lemma}
\begin{proof}
	This type of observation originates in \cite{BKN}.
	We calculate
	\begin{eqnarray*}
		\begin{aligned}
			&\frac{1}{1-\delta}|\omega|^{1+\delta}\triangle{}|\omega|^{1-\delta}
			\;=\;-\delta{}|\nabla|\omega||^2+|\eta|\triangle|\omega| \\ &\quad\quad\quad\quad\;=\;
			|\nabla\omega|^2-(1+\delta)|\nabla|\omega||^2+\left<\omega,\triangle\omega\right>.
		\end{aligned}
	\end{eqnarray*}
	Using $\delta=\frac12$ with (\ref{IneqImprovedKato}) gives $2|\omega|^{\frac32}\triangle|\omega|^{\frac12}\ge\left<\omega,\triangle\omega\right>$.
	Finally, use $\triangle\omega=\frac{s}{3}\omega$.
\end{proof}

\subsection{Convergence theory of half-conformally flat manifolds}

Tian-Viaclovsky undertook a systematic study of the convergence behavior of Bach-flat manifolds---a class of manifolds that includes half-conformally flat manifolds---assuming volume growth lower bounds in \cite{TV1}\cite{TV2}\cite{TV3}.
As is well-known, volume growth lower bounds are implied by a bound on the Sobolev constant, so we can certainly utilize these results.
For us, the three most useful results of their study, all quoted from \cite{TV3}, are
\begin{theorem}[Volume upper bounds in terms of volume lower bounds] \label{ThmTVVolUpper}
	Assume $(M^4,g)$ is a compact Bach-flat manifold with constant scalar curvature $s$.
	Assume there is a constant $K$ so $|\chi|$, $|\tau|$, and $max_{M^4}|s|$ are bounded by $K$, and so $Vol\,B_r(p)>\frac1Kr^4$ for all $p\in{M}^4$ and all $r$ for which $Vol\,B_r(p)<\frac12Vol(M^4)$.
	Then there is some $C'$ dependent only on $K$ so that for all $r>0$ we have
	\begin{eqnarray}
		Vol\,B_r(p)\;\le\;C'r^4.
	\end{eqnarray}
\end{theorem}
\begin{proof}
	Using (\ref{EqnRmBounds}) to get uniform bounds on $L^2(|\Riem|)$, we use theorem 1.3 of \cite{TV3} to obtain the result.
\end{proof}

\begin{theorem}[$\epsilon$-Regularity for Bach-flat manifolds]
	Assume $(M^4,g)$ is a compact Bach-flat manifold with constant scalar curvature $s$.
	Assume there is a constant $K$ so $|\chi|$, $|\tau|$, and $\max_{M^4}|s|$ are bounded by $K$, and so $Vol\,B_r(p)>\frac1Kr^4$ for all $p\in{M}^4$ and all $r>0$ for which $Vol\,B_r(p)<\frac12Vol(M^4)$.
	Then there is some $\epsilon_0>0$ so that $\int_{B_r(p)}|\Riem|^2<\epsilon_0$ implies
	\begin{eqnarray}
		\sup_{B_{r/2}(p)}|\nabla^k\Riem|\;\le\;C'r^{-2-k}\left(\int_{B_r(p)}|\Riem|^2\right)^{\frac12}
	\end{eqnarray}
	for some $C'=C'(K,k)$.
\end{theorem}
\begin{proof}
	Using \ref{EqnRmBounds} to uniformly bound $L^2(|\Riem|)$, this is theorem 1.2 of \cite{TV3}.
\end{proof}

\begin{theorem}[Convergence of Bach-flat manifolds] \label{ThmTVConvergence}
	Assume $(M_i^4,g)$ is a sequence of compact Bach-flat manifold with constant scalar curvature $s_i$.
	Assume there is a constant $\Lambda$ so
	\begin{eqnarray}
		L^2(|\Riem{}_i|)\le\;\Lambda \quad \text{and} \quad |s_i|<\Lambda.
	\end{eqnarray}
	Also assume $Vol\,B_r(p)>\frac1Kr^4$ for all $p\in{M}^4$ and all $r>0$ for which $Vol\,B_r(p)<\frac12Vol(M_i^4)$.
	Assume further that $Vol(M_i^4)=1$.
	
	Then a subsequence of the $(M_i^4,g_i)$ converges in the Gromov-Hausdorff sense to a 4-dimensional Riemannian multifold $(M_\infty^4,g_\infty)$ with $Vol(M_\infty^4)=1$, and where the number of multifold points is bounded is uniformly bounded in terms of $\Lambda$, and the multiplicity of each multifold point is bounded in terms of $\Lambda$.
\end{theorem}
\begin{proof}
	After using \ref{EqnRmBounds} to uniformly bound $L^2(|\Riem|)$, the convergence result is corollary 1.5 of \cite{TV3}.
	The statement that $Vol(M_\infty^4)=1$ follows from Theorem \ref{ThmTVVolUpper}, in the following way.
	The convergence $M_i^4\rightarrow{}M_\infty^4$ is uniform away from finitely many balls of fixed, but arbitrarily small radius around the multifold points.
	The sum of the volumes of these small balls, however, must be smaller than a multiple of the volumes of corresponding Euclidean balls, but the upper bound in Theorem \ref{ThmTVVolUpper}, which is uniformly small.
	In the limit, therefore, their volumes are zero and we have $C^\infty$ convergence outside these balls; therefore we retain continuity of volume as we pass to the limit in $M_i^4\rightarrow{}M_\infty^4$.
\end{proof}

\subsection{Linear analysis of closed sections of $\bigwedge^+$} \label{SubSecRegularityWithSobolev}

Having recalled the convergence theory of half-conformally flat manifolds of Tian-Viaclovsky, we turn to the convergence of closed sections of $\bigwedge^+$.
If $\omega\in\bigwedge^+$ and $d\omega=0$ then of course $d*\omega=0$ and so $\omega$ is Hodge-harmonic: $\triangle_d\omega=0$.
Equation (\ref{EqnMainBochner}) provides the linear elliptic equation $\triangle\omega=\frac{s}{3}\omega$, and Lemma \ref{LemmaImprovedElliptic} gives $\triangle\sqrt{|\omega|}\ge\frac{s}{6}\sqrt{|\omega|}$.
Assuming the Sobolev constant $C_S$ is bounded, the linear elliptic theory \cite{GT} gives the following $L^\infty$ bound.
\begin{proposition}[Regularity for $\eta$] \label{PropRegularityForEta}
	Suppose the ball $B_r$ has Sobolev constant $C_S$ in the sense of equation (\ref{IneqSobolev}).
	If $\omega\in\bigwedge^+$ is closed, then
	\begin{eqnarray}
		\begin{aligned}
			\sup_{B_{r/2}}|\omega|
			&\;\le\;C\cdot(C_S)^4\cdot{}e^{-sr^2/48}\cdot{}
			r^{-4}\cdot\int_{B_r}|\omega| \\
			&\;\le\;D\cdot{}e^{-sr^2/48}\cdot{}
			r^{-2}\cdot\left(\int_{B_r}|\omega|^2\right)^{\frac12} \label{IneqRegularityI}
		\end{aligned} \label{EqnEllipticBoundedness}
	\end{eqnarray}
	where $C$ is a universal constant and $D=D(C_S)$.
\end{proposition}
\begin{proof}
We expect the Moser iteration process is familiar to most readers, so we are brief; see \cite{GT} for a fuller account.
Let $u$ be any function with gradient and Laplacian defined in the distributional sense; then given any $C^1$ function $\varphi$, for $p>1$ we have
\begin{eqnarray}
	\int\varphi^2|\nabla{}u^p|^2\;\le\;4\left(\frac{p-1}{2p-1}\right)^2\int|\nabla\varphi|^2u^{2p}\,-\,2p\frac{p-1}{2p-1}\int\varphi^2u^{2p-1}\triangle{}u \label{IneqGradEst}
\end{eqnarray}
and for $p=1$ we have $\int\varphi^2|\nabla{}u|^2\;\le\;4\int|\nabla\varphi|^2u^2-2\int\varphi^2u\triangle{}u$.
The Sobolev inequality gives
\begin{eqnarray}
	\begin{aligned}
		\left(\int\varphi^4u^{4p}\right)^{\frac12}
		&\;\le\;2C_S^2\int|\nabla\varphi|^2u^{2p}
		\,+\,2C_S^2\int\varphi^2|\nabla{u}^p|^2.
	\end{aligned} \label{IneqSobolevIneq}
\end{eqnarray}
Setting $u=\sqrt{|\omega|}$ so $u\triangle{}u\ge\frac{s}{6}u^2$ and combining (\ref{IneqGradEst}) and (\ref{IneqSobolevIneq}) we obtain
\begin{eqnarray}
	\begin{aligned}
		\left(\int\varphi^4|\omega|^{2p}\right)^{\frac12}
		&\le4C_S^2\int|\nabla\varphi|^2|\omega|^{p}
		+\frac13C_S^2p(-s)\int\varphi^2|\omega|^{p}
	\end{aligned} \label{IneqIterationFirstStep}
\end{eqnarray}
for $p\ge1$.
Now we create an iteration proceedure.
Set $p_i=2^i$, and $r_i=\frac12(1+2^{-i})r$.
Let $\varphi_i$ be a cutoff function with $\varphi\equiv1$ in $B_{r_i}$, $\varphi_i\equiv0$ outside of $B_{r_{i-1}}$, and $|\nabla\varphi_i|<4/(r_{i-1}-r_i)=4\cdot{}2^i$
Then (\ref{IneqIterationFirstStep}) becomes
\begin{eqnarray}
	\left(\int_{B_{r_i}}|\omega|^{2^{i+1}}\right)^{\frac{1}{2^{i+1}}}
	\le
	\left(16C_S^2r^{-2}4^i\right)^{\frac{1}{2^i}}\left(1+2^{i}\frac{-sr^2}{48}\right)^{\frac{1}{2^i}}\left(\int_{B_{r_i}}|\omega|^{2^i}\right)^{\frac{1}{2^i}}. \label{IneqMoserSecond}
\end{eqnarray}
One can prove that $\prod_i\left(1+\frac{-sr^2}{48}2^i\right)^{\frac{1}{2^i}}\le{}Exp(-sr^2/48)$ and so (\ref{IneqMoserSecond}) iterates to
\begin{eqnarray}
	\begin{aligned}
		\left(\int_{B_{r_i}}|\omega|^{2^{N+1}}\right)^{\frac{1}{2^{N+1}}}
		&\le
		\left[\prod_{i=0}^N\left(16C_S^2r^{-2}4^i\right)^{\frac{1}{2^i}}\right]Exp[-sr^2/48]
		\int_{B_{r_0}}|\omega| \\
		&\le\;C\cdot{}C_S{}^4\cdot{}Exp[-sr^2/48]\cdot{}r^{-4}\cdot{}\int_{B_r}|\omega|.
	\end{aligned} \label{IneqMoserThird}
\end{eqnarray}
Sending $N\rightarrow\infty$ gives inequality (\ref{IneqRegularityI}).
\end{proof}

\begin{lemma}[Convergence of sections] \label{LemmaConvOfSections}
	Assume $(M_i^4,g_i)$ is a sequnce of half-conformally flat manifolds satifying the conditions of Theorem \ref{ThmMain}.
	Assume that each manifold $i$ has a collection of $k$ many closed sections $\{\omega_i^1,\dots,\omega_i^k\}$ of $\bigwedge^+$, and that they are $L^2$-orthonormal: $\left<\omega_i^j,\omega_i^k\right>_{L^2}=\delta^{jk}$.
	
	Then a subseqeunce of the manifolds converges in the Gromov-Hausdorff sense to a scalar-flat multifold $M^4_\infty$, and, possibly passing to a further subequence, the sections $\omega_i^j$ converge to a collection of closed sections $\{\omega_\infty^1,\dots,\omega_\infty^k\}$ on the smooth portion of $(M_\infty^4,g_\infty)$.
	Further, each $\omega_\infty^j$ remains uniformly bounded, and we retain the $L^2$-orthonormality of the sections: $\left<\omega_{\infty}^j,\omega_\infty^k\right>_{L^2}=\delta^{jk}$.
\end{lemma}
\begin{proof}
	This follows due to the uniform boundedness provided by (\ref{EqnEllipticBoundedness}) on balls of finite size, which prevents any $L^2$-energy of the sections from disappearing into any of the bubbles.
	
	Pick any point $p_i\in{}M_i^4$; we show that $|\omega_i^j|$ is uniformly bounded at $p_i$.
	To see this, let $r$ be the radius so that $Vol\,B_{p_i}(r)=\frac{1}{16}Vol\,M_i^4=\frac{1}{16}$.
	Recall that the Sobolev constant provides a uniform lower bound for the volue growth of balls, and that the Tian-Viaclovsky result, listed here as Theorem \ref{ThmTVVolUpper} provides a uniform upper bound for volume ratios.
	Therefore we can write
	\begin{eqnarray}
		C''r^4\;\le\;
		\frac{1}{16}=Vol\,B(p_i,r)\;\le\;C'r^4
	\end{eqnarray}
	where the constant $C''$ comes from the Sobolev constant bound and $C'$ is the Tian-Viaclovsky bound.
	Thus we have $r\ge(C')^{1/4}/2$ and $r\le(C'')^{1/4}/2$.
	Our regularity result Proposition \ref{PropRegularityForEta} now gives
	\begin{eqnarray}
		|\omega_i^j|_{p_i}\;\le\;De^{-s(C'')^{1/2}/96}\frac{4}{(C')^{\frac12}} \label{EllLInfBound}
	\end{eqnarray}
	where we used the normalization $|\omega_i^j|_{L^2}=1$.
	With $s\nearrow0$ we obtain a uniformly finite bound on $L^\infty(|\omega_i^j|)$.
	
	At any smooth point $p_\infty\in{}M_\infty^4$ of the limit, there is some ball of finite size $B=B(p_\infty,r)$ so that the closure of $B$ contains only smooth points.
	On such a ball the metrics converge: $g_i\rightarrow{}g_\infty$ in the $C^\infty$ sense, and because $d\omega_i^j=0$ is an elliptic condition, the forms $\omega_i^j$ also converge smoothly to a form $\omega_\infty^j$ on $B$.
	
	Covering the singular points of $M_\infty^4$ with balls of radius $\delta$.
	The Tian-Viaclovsky convergence theory states that there are uniformly finite many singular points, $q^1,\dots,q^N$, $N=N(\Lambda)$.
	Set $\mathcal{B}_\delta=\bigcup_{k=1}^NB(q^k,\delta)$.
	The Gromov-Hausdorff convergence theory provides diffeomorphisms
	\begin{eqnarray}
		\pi_{i,\delta}:M^4_i\setminus\Omega_{i,\delta}\;\longrightarrow\;M_\infty^4\setminus\mathcal{B}
	\end{eqnarray}
	along which the metric convergence is uniform in the $C^\infty$ sense.
	The sets $\Omega_{i,\delta}$ consist of small balls around bubble regions.
	By the Tian-Viaclovsky upper volume bound, we have that $Vol\Omega_{i,\delta}\le{}NC'\delta^4$.
	Combining this with the $L^\infty$ bound (\ref{EllLInfBound}) gives
	\begin{eqnarray}
		\int_{\Omega_{i,\delta}}|\omega_i^j|^2\,dVol_i
		\;\le\;NC'\delta^4\,\cdot\,De^{-s(C'')^{1/2}/96}\frac{4}{(C')^{\frac12}}
	\end{eqnarray}
	which means no $L^2$-energy of the sections $\omega_i$ can be absorbed into the bubbles.
	
	Finally because convergence of the sections $\omega_i^j\rightarrow\omega_\infty^j$ is uniform on the smooth portion $M_i^4\setminus\Omega_{i,\delta}$, we retain, for each $\delta$, that
	\begin{eqnarray}
		1\;\le\;\int_{M_\infty^\delta\setminus\mathcal{B}_\delta}|\omega_\infty^j|^2dVol_\infty
		\;\ge\;1-NC'\delta^4\,\cdot\,De^{-s(C'')^{1/2}/96}\frac{4}{(C')^{\frac12}}.
	\end{eqnarray}
	Sending $\delta\searrow0$, we see $|\omega_\infty^j|_{L^2}=1$.
	The fact that $\left<\omega_{\infty}^j,\omega_{\infty}^k\right>_{L^2}=0$ follows similarly from the $C^\infty$ convergence of the sections on the smooth part of the manifold.
\end{proof}

\noindent{\it Proof of Theorem \ref{ThmMain}} 

Using (\ref{LemmaConvOfSections}) we obtain a bound on the betti numbers $b^+$ of the manifolds $M_i^4$.
Pick a large integer $K<\infty$, and assume we have a sequence $b^+(M_i^4)\ge{K}$, where the Riemannian manifolds $(M^4_i,g_i)$ satisfy the hypotheses of Theorem \ref{ThmMain}.

Taking a Gromov-Hausdirff limit $(M^4_i,g_i)\rightarrow(M_\infty^4,g_\infty)$, we obtain a scalr-flat, half-conformally flat multifold $(M_\infty^4,g_i)$.
The Tian-Viaclovsky theory, Theorem \ref{ThmTVConvergence} limits the number and multiplicity of each multifold point, which means the number of orbifold components of $(M_\infty^4,g_\infty)$ is uniformly bounded by some number $N=N(\Lambda)$.

The LeBrun method of \cite{LeB} shows that on each compact orbifold component, sections $\omega_\infty^j\in\bigwedge^+$ with $d\omega_\infty^j$ are covariant-constant.
Becuase the rank of $\bigwedge^+$ is three, each orbifold component has at most 3 non-zero closed sections of $\bigwedge^+$.
Therefore the total number of closed sections of $\bigwedge^+$ that are non-zero somewhere is at most 3 times the number of orbifold components, which is bounded by $3N$

But our Lemma \ref{LemmaConvOfSections} guarantees $L^2$ continuity of all components $\omega_i^j$, $1\le{}j\le{}K$.
We conclude that $K\le{}2N$, so $b^+$ is bounded uniformly in terms of $\Lambda$, as claimed.
\hfill $\square$

\section{Examples} \label{SecExamples}

Of central importance to the study of sections $\omega\in\bigwedge^+$ is learning how they might behave within bubbles.
Here we construct several examples that illustrate certain of these behaviors.
In the non-collapsed setting, the singularity models are complete ALE manifolds with finitely many ends, along with closed, bounded sections of $\bigwedge^+$.
The phenomena we explore are:
\begin{enumerate}
	\item A 2-ended ALE manifold with a closed, bounded $\omega\in\bigwedge^+$ that is non-K\"ahler, but is asymptotically K\"ahler on both ends, and
	\item A multi-ended ALE bubble model with non-trivial cohomology at the second level, and has a closed, bounded $\omega\in\bigwedge^+$ that is asymptotically K\"ahler along one end, and asymptotically 0 along all other ends.
\end{enumerate}

To build our asymptotically K\"ahler closed sections on a certain 2-ended ALE manifold, we learn to construct closed sections of $\bigwedge^+$ using a dimension-reduction method.
Letting $t$ be a distance function that is smooth on some region, consider its level-sets $N_{t_0}^3=\{t=t_0\}$.
One notices that both $\bigwedge^1_{N_{t_0}^3}$ and $\bigwedge^+\Big|_{N_{t_0}^3}$ are three-dimensional.
In fact there is an isometric isomorphism between then, given by
\begin{eqnarray}
	\begin{aligned}
		&F_{t}:\bigwedge{}^+\Big|_{N_{t}^3}\;\longrightarrow\;\bigwedge{}^1{}_{N_{t}^3}, \\
		&F_{t}(\omega)\;=\;i_{\sqrt{2}dt}\omega
	\end{aligned} \label{EqnIsomIso}
\end{eqnarray}
where $i_{v}\eta=*(v\wedge*\eta)$ is the interior product.

If $t$ is considered a ``time'' variable, then a form $\omega\in\bigwedge^+$ is the same as a time-varying form $\eta(t)=F_t(\omega)$ on the time-varying manifolds $N^3_t$.
The issue is determining the conditions that forces $d\omega=0$.
This formulation allows a separation of variables approach to solving $d\omega=0$, where the ``time'' variable is separated from the ``space'' variables which exist on the manifolds $N^3_t$.

Below we do this for flat $\mathbb{R}^4$ where $t$ is the distance to the origin and $N^3_t$ is the 3-sphere of radius $t$.
This allows us to build examples of type (1).

{\bf Remark}.
The great difficulty with finding closed sections of $\bigwedge^+$ on ALE manifolds is the fact that the equation $d\omega=0$ is a first order, overdetermined, elliptic PDE.
Being overdetermined is actually not a serious issue; see the remark just after the proof of Proposition \ref{PropEvolCorrespond}.
The general problem of solving first order PDEs under the condition of being uniformly bounded is a very difficult problem.
To establish existence of a bounded solution to $d\omega=0$ on an ALE manifold, one would like to solve this equation on very large, but compact domains, and then take a limit.
Solving $d\omega=0$, $\omega\in\bigwedge^+$ on a half-conformally flat manifold-with-boundary is entirely analagous to solving the $\bar\partial$-problem on a complex manifold-with-boundary.
Determining whether any solutions exist at all is a complicated problem, and admissibility of boundary values is more complex yet.
Consider that, by the Cartan-K\"ahler theorem, the solution of $d\omega=0$ on the germ of a codimension-1 submanifold uniquely specifies the solution of $d\omega=0$ on its entire domain of definition.

\subsection{Separation of variables for $d\omega=0$ on flat $\mathbb{R}^4$} \label{SubSecSeparation}

Let $t:\mathbb{R}^4\rightarrow\mathbb{R}$ be the distance to the origin on flat $\mathbb{R}^4$.
The metric is
\begin{eqnarray}
	\begin{aligned}
		&g\;=\;dt^2\,+\,t^2g_{\mathbb{S}^3} \label{EqnFlatMetric}
	\end{aligned}
\end{eqnarray}
where $g_{\mathbb{S}^3}$ is the round metric on $\mathbb{S}^3$.
We solve $d\omega=0$ by separating the $t$ variable from the spherical variables.

First we record a bit of geometry on $\mathbb{S}^3$.
Let $\eta^1,\eta^2,\eta^3$ be the standard left-invariant unit frames on $\mathbb{S}^3$; recall that
\begin{eqnarray}
	d\eta^i\;=\;\epsilon^i{}_{jk}\eta^j\wedge\eta^k
\end{eqnarray}
where $\epsilon^i{}_{jk}$ is the Levi-Civita (totally antisymmetric) symbol---we have $d\eta^1=2\eta^2\wedge\eta^3$, for example.
Let $\{e_1,e_2,e_3\}$ be the corresponding frame.
Covector fields on $\mathbb{S}^3$ can be written $\eta=\alpha_i\eta^i$ where $\alpha_i:\mathbb{S}^3\rightarrow\mathbb{R}$.
We define the ``divergence'' and ``curl'' operators by
\begin{eqnarray}
	\begin{aligned}
		&div:\bigwedge{}^1_{\mathbb{S}^3}\;\rightarrow\;\bigwedge{}^0_{\mathbb{S}^3}, 
		\quad\;\; 
		div(\eta)\;\triangleq\;\delta^{ij}e_i(\alpha_i) \\
		&curl:\bigwedge{}^1_{\mathbb{S}^3}\;\rightarrow\;\bigwedge{}^1_{\mathbb{S}^3}, 
		\quad
		curl(\eta)\;\triangleq\;\epsilon_i{}^{ij}e_i(\alpha_j)\eta^i.
	\end{aligned}
\end{eqnarray}
These are the the familiar vector calculus operators, except executed in a frame on $\mathbb{S}^3$ rather than on flar 3-space.
But in fact these are invariant operators.
\begin{lemma}
	On round $\mathbb{S}^3$ we have
	\begin{eqnarray}
		\begin{aligned}
			&div(\eta)\;=\;*d*\eta \\
			&curl(\eta)\;=\;-2\eta\,+\,*d\eta.
		\end{aligned}
	\end{eqnarray}
\end{lemma}
\begin{proof}
	The Hodge star is
	\begin{eqnarray}
		\begin{aligned}
			&*:\bigwedge{}^1\rightarrow\bigwedge{}^2, \quad
			*\eta^i\;=\;\frac12\epsilon^i{}_{jk}\eta^j\wedge\eta^k \\
			&*:\bigwedge{}^2\rightarrow\bigwedge{}^1, \quad
			*\eta^j\wedge\eta^k\;=\;\epsilon_i{}^{jk}\eta^i.
		\end{aligned} \label{EqnHodgeStarS3}
	\end{eqnarray}
	For $div(\eta)$ we use (\ref{EqnHodgeStarS3}) along with the fact that $d(\eta^j\wedge\eta^k)=0$ to find
	\begin{eqnarray}
		\begin{aligned}
			*\eta&\;=\;*(\alpha_i\eta^i)\;=\;\frac12\epsilon^i{}_{jk}\alpha_i\eta^j\wedge\eta^k \\
			d*\eta&
			\;=\;\frac12\epsilon^i{}_{jk}e_l(\alpha_i)\eta^l\wedge\eta^j\wedge\eta^k \\
			*d*\eta&\;=\;\frac12\epsilon^i{}_{jk}e_l(\alpha_i)\epsilon^{ljk}
			\;=\;\delta^{il}e_l(\alpha_i)
			\;=\;div(\eta)
		\end{aligned}
	\end{eqnarray}
	where we used the identity $\epsilon^i{}_{jk}\epsilon^{ljk}=2\delta^{il}$.
	To compute $curl$, we use $d\eta^i=\epsilon^i{}_{jk}\eta^j\wedge\eta^k$ to find
	\begin{eqnarray}
		\begin{aligned}
			d\eta
			&\;=\;d\left(\alpha_j\eta^j\right)
			\;=\;e_i(\alpha_j)\eta^i\wedge\eta^j\,+\,\alpha_i\epsilon^i{}_{jk}\eta^j\wedge\eta^k \\
			*d\eta
			&\;=\;e_i(\alpha_j)\epsilon_k{}^{ij}\eta^k
			\,+\,\alpha_i\epsilon^i{}_{jk}\epsilon_l{}^{jk}\eta^l \\
			&\;=\;e_i(\alpha_j)\epsilon_k{}^{ij}\eta^k
			\,+\,2\alpha_i\delta^i_l\eta^l \\
			&\;=\;curl(\eta)\,+\,2\eta
		\end{aligned}
	\end{eqnarray}
	where we used $\epsilon^i{}_{jk}\epsilon_l{}^{jk}=2\delta^i_l$.	
\end{proof}
Considering $\left(\mathbb{R}^{>0}\times\mathbb{S}^3\right)\cup\{pt\}=\mathbb{R}^4$ with metric (\ref{EqnFlatMetric}), the forms $\eta^i$ pull back to forms on $\mathbb{R}^4\setminus\{pt\}$, where the unit forms are $t\eta^i$.
The three standard covariant-constant sections of $\bigwedge^+$ are
\begin{eqnarray}
	\omega^i=\frac{1}{\sqrt{2}}\left(tdt\wedge\eta^i\,+\,\frac12t^2\epsilon^i{}_{jk}\eta^j\wedge\eta^k \right) \label{EqnKahlerExplicit}
\end{eqnarray}
or equivalantly
\begin{eqnarray}
	\omega^i\;=\;F_t^{-1}(t\eta^i)
\end{eqnarray}
where $F_t$ is the isomorphism from (\ref{EqnIsomIso})
For instance $\omega^1=\frac{1}{\sqrt{2}}\left(tdt\wedge\eta^1+t^2\eta^2\wedge\eta^3\right)$.
It is easily checked from (\ref{EqnKahlerExplicit}) that $d\omega^i=0$ and $|\omega^i|^2=*(\omega^i\wedge\omega^i)=1$.
\begin{proposition} \label{PropEvolCorrespond}
	If $\omega=a_i\omega^i$ is a section of $\bigwedge^+$, then $d\omega=0$ is the same as the evolution equations on $\mathbb{S}^3$ given by
	\begin{eqnarray}
		\begin{aligned}
			&div(\eta)\;=\;0 \\
			&\frac{\partial\eta}{\partial{t}}\;=\;\frac1t{}curl(\eta) \\
			&\quad\;\;\;=\;\frac1t{}\left(*d\eta\,-\,2\eta \right)
		\end{aligned} \label{EqnEvoFirstEqn}
	\end{eqnarray}
	where the correspondance is via the linear isometry $\eta=F_t(t^{-1}\omega)=\alpha_i\eta^i$, and we define $\frac{\partial\eta}{\partial{t}}\triangleq\frac{\partial\alpha_i}{\partial{t}}\eta^i$.
\end{proposition}
{\bf Remark}.
Changing to $u=\log{}t$ the evolution equations are
\begin{eqnarray}
	div(\eta)\;=\;0, \quad\quad
	\frac{\partial\eta}{\partial{}u}\;=\;curl(\eta).
\end{eqnarray}
As discussed in the remarks following (\ref{EqnEuclMax}), these are the so-called Euclidean-Maxwell equations.
\begin{proof}
	With $d\omega^i=0$ we use $d\omega=d\alpha_i\wedge\omega^i$ to obtain
	\begin{eqnarray}
		\begin{aligned}
			d\omega
			&\;=\;\frac{\partial\alpha_i}{\partial{t}}dt\wedge\omega^i 
			\;+\;e_j(\alpha_i)\eta^j\wedge\omega^i.
		\end{aligned}
	\end{eqnarray}
	Because $\omega^i=\frac{1}{\sqrt{2}}\left(tdt\wedge{}\eta^i+\frac12t^2\epsilon^i{}_{st}\eta^s\wedge\eta^t \right)$, we have
	\begin{eqnarray}
		\begin{aligned}
			&dt\wedge\omega^i=\frac{1}{2\sqrt{2}}t^2\epsilon^i{}_{st}dt\wedge\eta^s\wedge\eta^t \\
			&\eta^j\wedge\omega^i=\frac{1}{\sqrt{2}}tdt\wedge\eta^i\wedge\eta^j+\frac{1}{2\sqrt{2}}t^2\epsilon^i{}_{st}\eta^j\wedge\eta^s\wedge\eta^t
		\end{aligned}
	\end{eqnarray}
	Therefore we obtain
	\begin{eqnarray}
		\begin{aligned}
			d\omega
			&\;=\;
			\frac{1}{2\sqrt{2}}\frac{\partial\alpha_i}{\partial{t}}
			t^2\epsilon^i{}_{st}dt\wedge\eta^s\wedge\eta^t \\
			&\quad\quad\;+\;e_j(\alpha_i)\left(
			\frac{1}{\sqrt{2}}tdt\wedge\eta^i\wedge\eta^j+\frac{1}{2\sqrt{2}}t^2\epsilon^i{}_{st}\eta^j\wedge\eta^s\wedge\eta^t\right) \\
			&\;=\;
			\frac{1}{\sqrt{2}}
			\left(\frac{\partial\alpha_i}{\partial{t}}t^2\frac12\epsilon^i{}_{st}\eta^s\wedge\eta^t
			+te_j(\alpha_i)\eta^i\wedge\eta^j
			\right)\wedge{}dt \\
			&\quad\quad\;+\;e_j(\alpha_i)
			\left(\frac{1}{2\sqrt{2}}t^2\epsilon^i{}_{st}\eta^j\wedge\eta^s\wedge\eta^t\right) \\
		\end{aligned}
	\end{eqnarray}
	To complete the computation, we utilize the Hodge-star on $\mathbb{R}^4$.
	We have $*(\eta^j\wedge\eta^s\wedge\eta^t)=(1/t^3)\epsilon^{jst}dt$ and $*(\eta^s\wedge\eta^t\wedge{}dt)=-(1/t)\epsilon^{st}{}_l\eta^l$.
	Therefore
	\begin{eqnarray}
		\begin{aligned}
			*d\omega
			&\;=\;
			-\frac{1}{\sqrt{2}}
			\left(\frac{\partial\alpha_i}{\partial{t}}t\frac12\epsilon^i{}_{st}\epsilon^{st}{}_l\eta^l
			+e_j(\alpha_i)\epsilon^{ij}{}_l\eta^l
			\right) \\
			&\quad\quad\quad\quad\;+\;e_j(\alpha_i)
			\left(\frac{1}{2\sqrt{2}t}\epsilon^i{}_{st}\epsilon^{jst}dt\right) \\
			&\;=\;
			-\frac{1}{\sqrt{2}}
			\left(t\frac{\partial\alpha_l}{\partial{t}}
			+e_j(\alpha_i)\epsilon^{ij}{}_l
			\right)\eta^l\;+\;\frac{1}{t\sqrt{2}}e_j(\alpha_i)\delta^{ij}dt \\
			&\;=\;
			-\frac{t}{\sqrt{2}}
			\left(\frac{\partial\alpha_l}{\partial{t}}\eta^l
			-\frac{1}{t}curl(\eta)
			\right)\eta^l
			\;+\;\frac{1}{t\sqrt{2}}div(\eta)\,dt \\
		\end{aligned}
	\end{eqnarray}
	We conclude that $d\omega=0$ precisely when $div(\eta)=0$ and $\frac{\partial\eta}{\partial{t}}-\frac1t{}curl(\eta)=0$.
\end{proof}

{\bf Remark}.
The expression $d\omega=0$ is an overdetermined first-order elliptic equation; nevertheless it is reducible to a critically determined equation.
We show how this is done on flat $\mathbb{R}^4$.
Expressing $d\omega=0$ on $\mathbb{R}^4$ as an evolution equation on $\mathbb{S}^3$, we have seen that $\eta=F_t(\omega)$ satisfies
\begin{eqnarray}
	div(\eta)=0, \quad \left(t\frac{\partial}{\partial{t}}-*d\,+\,2\right)\eta=0 \label{EqnEvoSystem}
\end{eqnarray}
where $\eta=F_t(\omega)$ is a time-varying covector field in $\bigwedge^1_{\mathbb{S}^3}$.
To see this is overdetermined, notice (\ref{EqnEvoSystem}) has four differential identities, whereas $\bigwedge^1_{\mathbb{S}^3}$ is only rank 3.
We reduce it to a critically determined evolution equation by utilizing the Hodge decomposition.
Because $\mathbb{S}^3$ is simply connected it has no harmonic 1-forms and its $\bigwedge^1$ Hodge decomposition is
\begin{eqnarray}
	\bigwedge{}^1\;=\;d\left(\bigwedge{}^0\right)\,\oplus\,d^*\left(\bigwedge{}^2\right).
\end{eqnarray}
Clearly the evolution equation $t\frac{\partial}{\partial{t}}-*d+2$ preserves the subspace $d^*\left(\bigwedge{}^2\right)$.
Therefore restricting to the closed subspace $d^*\left(\bigwedge{}^2\right)$, the overdetermined system (\ref{EqnEvoSystem}) reduces to the critically determined system
\begin{eqnarray}
	\left(t\frac{\partial}{\partial{t}}-*d\,+\,2\right)\eta\;=\;0.
\end{eqnarray}

\begin{proposition}[Separation of Variables]
	Assume $\eta'=\alpha_i\eta^i$ is a time-varying field on $\mathbb{S}^3$.
	To separate out the time variable, express $\alpha_i=\beta\gamma_i$ where $\beta=\beta(t)$ and $\gamma_i\in\bigwedge^0_{\mathbb{S}^3}$ has no $t$-dependency.
	Writing $\eta=\beta\eta'$ where $\eta'=\gamma_i\eta^i$, (\ref{EqnEvoFirstEqn}) is 
	\begin{eqnarray}
		\begin{aligned}
			&div\left(\eta'\right)\;=\;0 \\
			&\left(t\frac{d}{dt}\log\beta\right)\cdot\eta'
			\;=\;\left(*d-2\right)\eta'.
		\end{aligned} \label{EqnSeparated}
	\end{eqnarray}
\end{proposition}
\begin{proof}
	Straightforward computation.
\end{proof}

Proceeding in the normal way, if $\eta_\lambda$ is an eigen-covector field for $*d$ on $\mathbb{S}^3$, meaning $*d\eta_\lambda=\lambda\eta_\lambda$, then (\ref{EqnSeparated}) is
\begin{eqnarray}
	&\left(t\frac{d}{dt}\log\beta\right)\eta_\lambda
	\;=\;(\lambda-2)\eta_\lambda
\end{eqnarray}
which holds if and only if $t(\log\beta)'=(\lambda-2)$, or $\beta=Ct^{\lambda-2}$.
Then
\begin{eqnarray}
	\eta\;=\;\sum_{\lambda\in{}spec(*d)}C_\lambda{}t^{\lambda-2}\eta_\lambda
\end{eqnarray}
is a solution of the evolution equation (\ref{EqnEvoSystem}).

Fortunately the eigenspace decomposition of $*d$ has already been accomplished, by Folland \cite{F} and Sandberg \cite{Sand}, although it was not expressed in precisely this way in either work.
Both works also contain a minor error.
Referencing those works, the eigenvalues of the operator 
$$
	*d:d^*\left(\bigwedge{}^2\right)\rightarrow{}d^*\left(\bigwedge{}^2\right)
$$
on the round sphere are the integers $\{\pm2,\pm3,\pm4,\dots\}$.
The error in \cite{F} and \cite{Sand} is that both claim that $\pm1$ are also eigenvalues, but we show in Theorem \ref{ThmEigenvaluesRestriction} below that this is not possible.
In fact neither \cite{F} nor \cite{Sand} ever bother compute the multiplicity of eigenvalues, even though they find complete eigenspace decompositions.
Had they done so, they would have found that $mult(\lambda)=\lambda^2-1$, which means obviously means $\lambda=\pm1$ cannot be an eigenvalue.
More concretely, formula (3.4c) of \cite{F} produces an absurdity when $\nu=1$ and $k=1$, and formula (24) of \cite{Sand} produces a triviality when $\lambda=\pm1$.
Formally we rule out $\lambda=\pm1$ in Theorem \ref{ThmEigenvaluesRestriction} using the improved elliptic inequality (\ref{LemmaImprovedElliptic}) for closed forms in $\bigwedge^+$.

Next we use the Cartan-K\"ahler theorem to show that solutions of $d\omega=0$ can always be expressed in series form.
\begin{proposition} \label{PropCorrespondance}
	Any solution of the evolution equations (\ref{EqnEvoSystem}) of the form
	\begin{eqnarray}
		\eta\;=\;\sum_{\lambda\in\mathbb{Z}\setminus\{-1,0,1\}}C_\lambda{}t^{\lambda-2}\eta_\lambda
	\end{eqnarray}
	produces a solution
	\begin{eqnarray}
		\begin{aligned}
			\omega
			&\;=\;\sum_{\lambda\in\mathbb{Z}\setminus\{-1,0,1\}}C_\lambda{}t^{\lambda-2}\omega_\lambda
		\end{aligned} \label{EqnOmegaAsSeriesSoln}
	\end{eqnarray}
	of $d\omega=0$, where we define $\omega_{\lambda}\triangleq{}F^{-1}_t(t\eta_\lambda)$.
	
	Conversely, assuming $d\omega=0$, $\omega\in\bigwedge^+_{\mathbb{R}^4}$ and $\omega$ is non-singular in some neighborhood of the unit sphere, then we can express $\omega$ as a series
	\begin{eqnarray}
		\omega\;=\;\sum_{\lambda\in\mathbb{Z}\setminus\{-1,0,1\}}C_\lambda{}t^{\lambda-2}\omega_\lambda
	\end{eqnarray}
	which holds in some neighborhood of the unit sphere.
\end{proposition}
\begin{proof}
	The first assertion is immediate from the separation of variables proceedure above, combined with Proposition \ref{PropEvolCorrespond}, which gives the equivalence of solutions of $d\omega=0$ and solutions of the evolution equation (\ref{EqnEvoSystem}).
	
	For the second assertion, we use the fact that the eigenspace decomposition of $*d:d^*\bigwedge^2_{\mathbb{S}^3}\rightarrow{}d^*\bigwedge^2_{\mathbb{S}^3}$ is complete (indeed, complete sets of eigen-covector fields are given in \cite{F} and \cite{Sand}).
	Let $\eta(1)=F_1(\omega)$ be the covector field associated to $\omega$, when restricted to the unit sphere.
	By assumption $\eta$ is smooth; it therefore has an eigenspace decomposition:
	\begin{eqnarray}
		\eta(1)\;=\;\sum_{\lambda\in\mathbb{Z}\setminus\{-1,0,1\}}C_\lambda{\eta_\lambda}.
	\end{eqnarray}
	Because $\eta$ is smooth, the standard theory states that the coefficients $C_\lambda$ are quickly decreasing (faster than polynomial) as $\lambda\rightarrow\pm\infty$.
	We have a corresponding solution of the Euclidean-Maxwell equations
	\begin{eqnarray}
		\eta(t)\;=\;\sum_{\lambda\in\mathbb{Z}\setminus\{-1,0,1\}}C_\lambda{t}^{\lambda-2}{\eta_\lambda}
	\end{eqnarray}
	provided this sequence converges.
	But this series {\it does} converge for $t$ in some range $t\in(1-\delta,1+\delta)$, as a consequence of the coefficients $C_\lambda$ decreasing rapidly.
	We have a corresponding solution
	\begin{eqnarray}
		\bar{\omega}\;=\;\sum_{\lambda\in\mathbb{Z}\setminus\{-1,0,1\}}C_\lambda{t}^{\lambda-2}{\omega_\lambda} \label{EqnSeriesForBarOmega}
	\end{eqnarray}
	for $d\bar{\omega}=0$, $\omega\in\bigwedge^+$, that exists in some neighborhood of the unit sphere.
	
	Now consider $\omega-\bar\omega$.
	By the convergence of the series (\ref{EqnSeriesForBarOmega}) on the unit sphere, we have that $\omega-\bar\omega=0$ on the unit sphere.
	But the unit sphere is a codimension-1 submanifold of $\mathbb{R}^4$.
	Therefore the uniqueness part of Cartain-K\"ahler theorem asserts that $\omega-\bar\omega\equiv0$ on $\mathbb{R}^4$ wherever the series for $\bar\omega$ converges.
\end{proof}

\begin{lemma}[$L^2$-Orthogonality] \label{LemmaOrthProp}
	Referencing the correspondance of Proposition \ref{PropCorrespondance}, assume $\omega_{\lambda_1}$, $\omega_{\lambda_2}$ are two solutions of $d\omega=0$ corresponding to eigen-covector fields $\eta_{\lambda_1}$, $\eta_{\lambda_2}$ on $\mathbb{S}^3$, where we assume $\eta_{\lambda_1}$, $\eta_{\lambda_2}$ are orthonormal.
	
	Then on any spherical shell $t=const$ we have the $L^2$-orthogonality property for $\omega_{\lambda_1}$, $\omega_{\lambda_2}$ as well:
	\begin{eqnarray}
		\int_{t=const}i_{dt}\left(\omega_{\lambda_1}\wedge{}\omega_{\lambda_2}\right)\;=\;0.
	\end{eqnarray}
	On any ball $t\le{}const$ we have
	\begin{eqnarray}
		\int_{t\le{}const}\omega_{\lambda_1}\wedge{}\omega_{\lambda_2}\;=\;0.
	\end{eqnarray}
\end{lemma}
\begin{proof}
	The first equality follows immediately from the $L^2$-orthogonality of eigen-covector fields of $*d$ on $\mathbb{S}^3$.
	
	The second equality follows from
	\begin{eqnarray}
		\int_{t\le{}const}\omega_{\lambda_1}\wedge{}\omega_{\lambda_2}
		\;=\;
		\int_{\tau=0}^{const}\left(\int_{t=\tau}i_{dt}\left(\omega_{\lambda_1}\wedge{}\omega_{\lambda_2}\right)\right)\,dt.
	\end{eqnarray}
\end{proof}

\begin{theorem}[Spectral gap on $\mathbb{S}^3$]\label{ThmEigenvaluesRestriction}
	On round $\mathbb{S}^3$, if $\lambda$ is an eigenvalue of $*d:d^*\bigwedge^2\rightarrow{}d^*\bigwedge^2$, then $|\lambda|\ge2$.
	If $\lambda=\pm2$ then any corresponding eigen-covector field $\eta_\lambda$ has constant norm $|\eta_\lambda|=const$ on $\mathbb{S}^3$.
	
	If $\omega_2\in\bigwedge^+$ is the form corresponding to an eigenvalue $+2$ covector field $\eta_2$, then $\omega_2$ is K\"ahler.
	
	If $\mu$ is an eigenvalue of $\triangle_d=dd^*+d^*d:d^*\bigwedge^2\rightarrow{}d^*\bigwedge^2$, then $\mu\ge4$.
\end{theorem}
{\bf Remark}. Although $*d\eta=\pm2\eta$ means $|\eta|=const$ and that $\triangle_d=(*d)^2$, it is not true that $\triangle_{d}\eta=4\eta$ means $|\eta|=const$.
\begin{proof}
	Assume $*d\eta_\lambda=\lambda\eta_\lambda$ and let $\omega_\lambda\in\bigwedge^+_{\mathbb{R}^4}$ be the corresponding closed 2-form.
	Using (\ref{EqnOmegaAsSeriesSoln}) we have that $\omega_\lambda=t^{\lambda-2}\eta_\lambda$ solves $d\omega_\lambda=0$, and of course $|\omega_\lambda|=t^{\lambda-2}|\eta_\lambda|$.
	
	We have the improved Kato inequality $\triangle\sqrt{|\omega_\lambda|}\ge0$ from Lemma \ref{LemmaImprovedElliptic}, where $\triangle$ is the rough Laplacian.
	An elementary computation gives
	\begin{eqnarray}
		\begin{aligned}
			0&\;\le\;\triangle|\omega_\lambda|^{\frac12}
			\;=\;\triangle\left(t^{-1+\lambda/2}|\eta_\lambda|^{\frac12}\right) \\
			&\;=\;
			\left(\triangle{}t^{-1+\lambda/2}\right)|\eta_\lambda|^{\frac12}
			+t^{-1+\lambda/2}\triangle|\eta_\lambda|^{\frac12} \\
			&\;=\;
			\left(
			(-1+\lambda^2/4)t^{-3+\lambda/2}
			\right)|\eta_\lambda|^{\frac12}
			+t^{-1+\lambda/2}\triangle|\eta_\lambda|^{\frac12} \\
		\end{aligned}
	\end{eqnarray}
	and therefore we have
	\begin{eqnarray}
		\begin{aligned}
			\left(1-\lambda^2/4\right)\;\le\;
			t^2|\eta_\lambda|^{-\frac12}\triangle|\eta_\lambda|^{\frac12}
		\end{aligned} \label{EqnKatoLaplacian}
	\end{eqnarray}
	which holds in a pointwise sense whenever $|\eta_\lambda|\ne0$.
	The eigen-covector field $\eta_\lambda$ is certainly smooth on $\mathbb{S}^3$, and so $|\eta|^{\frac12}$ obtains a maximum somewhere on $\mathbb{S}^3$.
	There, the maximum principle gives $\triangle|\eta_\lambda|^{\frac12}\le0$.
	This forces $1-\lambda^2/4\le0$, which means $4\le\lambda^2$.
	This concludes the proof that if $*d\eta_\lambda=\lambda\eta_\lambda$, then $|\lambda|\le\pm2$.
	
	To prove that a corresponding eigen-covector field $\eta_\lambda$, $\lambda=\pm1$, has $|\eta_\lambda|=Const$, note that the Laplacian on $\mathbb{R}^4$ splits: $\triangle_{\mathbb{R}^4}=t^{-3}\frac{\partial}{\partial{t}}\left(t^3\frac{\partial}{\partial{t}}\right)+t^{-2}\triangle_{\mathbb{R}^3}$.
	Then from (\ref{EqnKatoLaplacian}), using the fact that $|\eta_\lambda|$ is $t$-invariant, we obtain
	\begin{eqnarray}
		0
		\;\le\;t^2\triangle_{\mathbb{R}^4}|\eta_\lambda|^{\frac12}
		\;=\;\triangle_{\mathbb{S}^3}|\eta_\lambda|^{\frac12}.
	\end{eqnarray}
	Since a continuous subharmonic function on a compact manifold is constant, we have $|\eta_\lambda|=const$.
	From (\ref{EqnOmegaAsSeriesSoln}) we have $|\omega_2|=const$ and $|\omega_{-2}|=const\cdot{}t^{-4}$.
		
	Using $\omega_2=const$ on $\mathbb{R}^4$ we verify that any $\omega_2$ must be K\"ahler.
	Letting $\varphi$ be any test function, we use $\triangle\omega=0$ and integration by parts to obtain
	\begin{eqnarray}
		\begin{aligned}
			\int\varphi^2|\nabla\omega|^2
			\;=\;-2\int\varphi\left<\nabla\varphi\otimes\omega,\,\nabla\omega\right>
			\;=\;-\int\varphi\left<\nabla\varphi,\,\nabla|\omega|^2\right>
			\;=\;0
		\end{aligned}
	\end{eqnarray}
	We conclude that $\nabla\omega\equiv0$, so $\omega$ is K\"ahler.
	
	Next we consider the assertion that $\triangle_d\eta_\mu=\mu\eta_\mu$, $\eta_\mu\in{}d^*\bigwedge^2$, implies $\mu\ge4$.	
	Certainly $\mu>0$, since the Hodge Laplacian is positive definite on $\mathbb{S}^3$.
	From $div(\eta_\mu)=d^*(\eta_\mu)=0$ we see that $\triangle_d=(d^*+d)^2=(*d)^2$.
	
	Now create the 1-form $\gamma^{\pm}_\mu=\sqrt{\mu}\eta_\mu\pm*d\eta_\mu$.
	Then we compute
	\begin{eqnarray}
		\begin{aligned}
			*d\gamma^{\pm}_\mu
			&\;=\;\sqrt{\mu}*d\eta_\mu\pm(*d)^2\eta_\mu \\
			&\;=\;\sqrt{\mu}*d\eta_\mu\pm\mu\eta_\mu \\
			&\;=\;\pm\sqrt{\mu}\left(\sqrt{\mu}\eta_\mu\pm*d\eta_\mu\right)
			\;=\;\pm\sqrt{\mu}\gamma^{\pm}_\mu.
		\end{aligned}
	\end{eqnarray}
	This shows that if $\triangle_d$ has an eigenvalue of $\mu$ then $*d$ has an eigenvalue of $\pm\sqrt{\mu}$.
	Becuase $\sqrt{\mu}\ge2$ by the previous result, we conclude $\mu\ge4$.
\end{proof}

{\bf Remark}.
The eigen-covector field $\eta_2$, where $*d\eta_2=2\eta_2$, produces a closed form $\omega_2\in\bigwedge^+$ with both $|\eta_2|=const$ on $\mathbb{S}^3$ and $|\omega_2|=const$ on $\mathbb{R}^4$.
Any such form is a K\"ahler form on $\mathbb{R}^4$ for the metric $g$.
(Incidentally, this confirms that the multiplicity of the $\lambda=2$ eigenvalue of $*d$ is three.)

For the eigenvalue $-2$ of $*d$, we have $|\eta_{-2}|=const$ on $\mathbb{S}^3$ but the corresponding closed field $\omega_{-2}\in\bigwedge^+$ does not have constant norm: $|\omega_{-2}|=const\cdot{}t^{-4}$, as seen from (\ref{EqnOmegaAsSeriesSoln}).
This is the K\"ahler form for the conformally related metric $\hat{g}=t^{-4}g$, which is a flat metric with the origin and infinity changing places.

{\bf Definition}. ({\it Asymptotically K\"ahler})
	Assume $(M^4,g)$ is a complete ALE manifold with a closed 2-form $\omega\in\bigwedge^+$, $d\omega=0$.
	Let $\Omega$ be an end of $M^4$, where $\Omega$ is diffeomorphic to $(\mathbb{R}^4\setminus{}B_1)/\Gamma$, where $\Gamma$ is some discrete subgroup of $SO(4)$.
	The end $\Omega$ is called {\it asymptotically K\"ahler} with respect to $\omega$ if the following holds:
	scaling the metric $g$ on $\Omega$ to create $g_\epsilon=\epsilon^2g$ and scaling $\omega$ to create $\omega_\epsilon=\epsilon^2\omega$, we have as $\epsilon\searrow0$ that $g_\epsilon$ converges in the $C^1$ sense to a flat metric $g_\infty$ on $\mathbb{R}^4/\Gamma$, and $\omega_\epsilon$ converges in the $C^1$ sense to a non-zero covariant-constant form $\omega_\infty$, as measured in the $g_\infty$ metric.

\begin{theorem} \label{PropDecayGapOfForms}
	Assume $\Omega\approx(\mathbb{R}^4\setminus{}B_1)/\Gamma$ with metric $g$ is an ALE manifold end of order at least $2$, and assume $\omega\in\bigwedge^+$ is closed and bounded.
	Then either 
	\begin{itemize}
		\item[{\it{i}})] $|\omega|=O(1)$ and $\Omega$ is asymptotically K\"ahler with respect to $\omega$, or else 
		\item[{\it{ii}})] $|\omega|=O(\rho^{-4})$ where $\rho$ is the distance function from $\partial\Omega$, or equivalently $\omega$ is bounded under the natural compactification of $\Omega$.
	\end{itemize}
\end{theorem}
{\bf Remark}. In part this is a ``gap theorem'' for the decay rate of any closed form $\omega\in\bigwedge^+$, assuming $|\omega|$ is bounded.
Decay rates strictly between $O(1)$ and $O(\rho^{-4})$ are forbidden.

{\bf Remark}.
The ``natural'' conformal compactification of a manifold end is $\hat{g}=G^{2}g$ where $G$ is the harmonic function on $\Omega$ with $G=1$ on $\partial\Omega$ and $G\rightarrow0$ at infinity.

{\bf Remark}.
In dimension 4, if a manifold end is ALE of order $2$ then its natural compactification is a Riemannian manifold with $L^\infty$ curvature tensor, and therefore at least $C^{1,\alpha}$ metric across the compactification point.
If the metric is special, say half-conformally flat or Bach-flat, the compactified metric is $C^\infty$ across the compactification point.
If the metric is scalar flat, the compactified metric remains scalar flat.
\begin{proof}
	Assume $|\omega|_g$ is bounded.
	The Harmonic function $G$ is asymptotically $\rho^{-2}$ so the compactified metric is $g_c\approx{}\rho^{-4}g$, and $|\omega|_{g_c}=\rho^{4}|\omega|_{g}$.
	If $r$ is the distance to the origin in the $g_c$ metric, then $r=\rho^{-1}+O(\rho^{-2})$, so that $\omega$ has a pole: $|\omega|_{g_c}=O(r^{-4})$.
	
	But by Lemma \ref{LemmaImprovedElliptic}, which is a consequence of the improved Kato inequality, we also have $\triangle\sqrt{|\omega|_{g_c}}\ge\frac{s_c}{6}$, where of course $\sqrt{|\omega|_{g_c}}=O(r^{-2})$.
	The usual elliptic theory says $\sqrt{|\omega|_{g_c}}$ has at worst a simple pole: $\sqrt{|\omega|_{g_c}}=Cr^{-2}+O(1)$.
	
	If $C=0$ then $|\omega|_{g_c}$ is bounded and so in the original metric $|\omega|_g$ decays like $O(\rho^{-4})$; this is case ({\it{ii}}).
	
	If $C\ne0$, then the fact that $|\omega|_{g_c}=Cr^{-4}+O(r^{-2})$ means $|\omega|_g=r^4|\omega|_{g_c}=C+O(\rho^{-2})$.
	Our task now is to show that $|\omega|_g$ asymptotically constant means $\omega$ is actually asymptotically K\"ahler.
	Take a blowdown limit of $(\Omega,g)$ by scaling $g'=\epsilon^2g$, $N\rightarrow\infty$, and simultaneously scale the form $\omega$ to obtain $\omega'=\epsilon^{-2}\omega$.
	Then $|\omega'|_{g'}=|\omega|_g$, so we retain $|\omega'|=O(1)$ for every $\epsilon$.
	The elliptic equation $\triangle'\omega'=s'/3$ forces $\omega'$ to converge in the limit and we retain $|\omega'|=O(1)$ in the limit.
	
	Proposition \ref{PropCorrespondance} says $\omega'$ is equal to its series representation.
	We write
	\begin{eqnarray}
		\omega'
		=\sum{}_{\lambda\in\mathbb{Z}\setminus\{-1,0,1\}}C_\lambda{}t^{\lambda-2}\omega_\lambda. \label{EqnFourierOmega}
	\end{eqnarray}
	We divide the series into three parts:
	\begin{eqnarray}
		\begin{aligned}
			\omega'&\;=\;
			C_2\omega_2
			\,+\,\sum_{\alpha=1}^\infty{}C_{\alpha+2}{}t^\alpha\omega_{\alpha+2}
			\,+\,\sum_{\alpha=4}^{\infty}C_{-\alpha+2}{}t^{-\alpha}\omega_{-\alpha+2}
		\end{aligned}
	\end{eqnarray}
	The ``positive'' series $\sum_{\alpha>0}C_{\alpha+2}{}t^\alpha\omega_{\alpha+2}$ converges for all $t\le1$ including $t=0$, whereas the second series has a singularity at $t=0$.
	But because 
	\begin{eqnarray}
		\begin{aligned}
			\sum_{\alpha=4}^{\infty}C_{-\alpha+2}{}t^{-\alpha}\omega_{-\alpha+2}
			\;=\;\omega
			\,-\,C_2\omega_2
			\,-\,\sum_{\alpha=1}^\infty{}C_{\alpha+2}{}t^\alpha\omega_{\alpha+2}
		\end{aligned}
	\end{eqnarray}
	is smooth at $t=0$, all negative coefficients are zero: $C_{-\alpha+2}=0$, $\alpha>4$.
	Thus
	\begin{eqnarray}
		\begin{aligned}
			\omega&\;=\;
			C_2\omega_2
			\,+\,\sum_{\alpha=1}^\infty{}C_{\alpha+2}{}t^\alpha\omega_{\alpha+2}.
		\end{aligned}
	\end{eqnarray}
	However this new sum has a pole of some order at $t=\infty$, which $\omega$ does not have, and therefore necessarily $C_{\alpha+2}=0$ for all $\alpha>0$.
	We have proven that $\omega=C_2\omega_2$.
	Since $\omega_2$ is K\"ahler by Theorem \ref{ThmEigenvaluesRestriction}, the claim follows.
\end{proof}

\subsection{Example: 2-ended asymptotically K\"ahler manifold} \label{SubSec2EndedAK}

On taking $s\nearrow0$ and finding a limit of our non-collapsed half-confomally flat manifolds, we know the closed section $\omega\in\bigwedge^+$ has very good behavior on every component of the limit.
The central problem is that $\omega$ might have bad behavior in the bubbles.
We present the examples of this section and the next to show that, since non-uniform behavior of $\omega$ cannot be ruled out in the bubbles, Theorem \ref{ThmMain} is the best that can be hoped for in general.

One possibility is presented in figure .
In this possibility, just before the limit is reached we have two larger manifolds connected by a tiny bubble that is a 2-ended flair.
When the limit is reached, we potentially have 2 scalar-flat manifolds connected at a point.
Each of the two components of the limit is scalar-flat, and may have up to 3 closed sections of $\bigwedge^+$.

One might wish to rule out two possibilities.
The first that a closed section $\omega\in\bigwedge^+$ might converge, as $s\nearrow0$, to a K\"ahler form on one manifold by not the other, via a transition within the bubble---such a transition is demonstrated by the $\omega$ defined by (\ref{EqnAKForm}) below.
Also possible is a closed section $\omega\in\bigwedge^+$ that might converge, as $s\nearrow0$, to a K\"ahler form on both manifold, but nevertheless not become globally K\"ahler due to bad behavior within the bubble---such ``bad behavior'' within the bubble is indeed possible, as we see again from (\ref{EqnAKForm}) by using $\alpha=\pm\beta$.

To construct our 2-ended example manifold, consider the metric on $\mathbb{R}^4\setminus\{pt\}$
\begin{eqnarray}
	g_\epsilon\;=\;\left(\epsilon^2+t^{-2}\right)^2g \label{Eqn2EndedMetric}
\end{eqnarray}
where $g$ is the flat metric given in (\ref{EqnFlatMetric}).
The usual conformal change formula confirms that $scal=0$, and of course $W^+=W^-=0$.
For each $\epsilon$ the metric $g_\epsilon$ is a 2-ended asymptotically Euclidean (AE) manifold.
The 3-sphere at $t=\epsilon^{-1}$ is a minimial separating surface, and has area $16\pi^2\epsilon^3$; see figure 2.

The function $t$ is a distance function for $g$ but not for $g_\epsilon$.
To express the $g_\epsilon$ more naturally, we create a new distance function $\rho$ which is the (signed) distance to the minimal seprating surface just described.
We compute
\begin{eqnarray}
	\begin{aligned}
		\rho=\rho(t)\;=\;\int_{\epsilon^{-1}}^\tau(\epsilon^2+\tau^{-2})d\tau
		\;=\;\epsilon^2t\,-\,t^{-1}. \label{EqnDistFunDFlare}
	\end{aligned}
\end{eqnarray}
\begin{lemma}
	The manifold $\mathbb{R}^4\setminus\{pt\}$ with the metric (\ref{Eqn2EndedMetric}), $\epsilon\ne0$, is 2-ended and ALE.
	The metric is scalar-flat and conformally flat.
	The metric is ALE of order 2, and as $\epsilon\searrow0$ the Gromov-Hausdorff convergence is to the one-point union of two copies of flat $\mathbb{R}^4$.
	For any $\epsilon>0$ we have	
	 $|\Ric{}_\epsilon|^2=\frac{192\epsilon^4}{(\rho^2+4\epsilon^2)^4}$ so $|\Riem_\epsilon|=|\Ric_\epsilon|=O(\rho^{-4})$.
\end{lemma}
\begin{proof}
	From (\ref{Eqn2EndedMetric}) and (\ref{EqnDistFunDFlare}) we compute
	\begin{eqnarray}
		\begin{aligned}
			g_\epsilon
			&\;=\;(\epsilon^2+t^{-2})^2dt^2\,+\,t^2(\epsilon^2+t^{-2})^2g_{\mathbb{S}^3} \\
			&\;=\;d\rho^2\,+\,\left(\rho^2\,+\,4\epsilon^2\right)g_{\mathbb{S}^3}.
		\end{aligned} \label{EqnMetricsOnR4}
	\end{eqnarray}
	Thsi clearly exhibits $g_\epsilon$ asymptotically as $g_\epsilon=d\rho^2+\rho^2(1+O(\rho^{-2}))g_{\mathbb{S}^3}$ and so $g$ is ALE of order $2$.
	
	To see that the Gromov-Hausdrff convergence is to the join of two copies of flat $\mathbb{R}^4$, notice that as $\epsilon\rightarrow0$ the expression $g_\epsilon=d\rho^2+(\rho^2+4\epsilon^2)g_{\mathbb{S}^3}$, with range $-\infty<\rho<\infty$, converges to $g_0=d\rho^2+\rho^2g_{\mathbb{S}^3}$, still with range $-\infty<\rho<\infty$.
	This identifies the sphere at $\rho=0$ to a point, which joins the $0<\rho<\infty$ copy of $\mathbb{R}^4$ with the $-\infty<\rho<0$ copy.
	
	If $\hat{g}=u^2g$, we recall the usual conformal change formula in dimension 4:
	\begin{eqnarray}
		\widehat{\Ric}\;=\;\Ric
		-\left(\frac{1}{u}\triangle{}u\right)g
		-\frac{2}{u}\nabla^2u\,+\,4\nabla\log{}u\otimes\nabla\log{}u-|\nabla\log{}u|^2g
	\end{eqnarray}
	which can be found, for instance, in \cite{Besse}.		
	Using $u=\epsilon^2+t^{-2}$ we have $\triangle{}u=0$, and further computation gives
	\begin{eqnarray}
		\begin{aligned}
			\Ric{}_\epsilon
			&\;=\;
			\frac{-4\epsilon^2}{t^4(\epsilon^2+t^{-2})^2}\left(
			4\nabla{}t\otimes{}\nabla{}t\,-\,g
			\right).
		\end{aligned} \label{EqnRicMod}
	\end{eqnarray}
	Tracing with $g_\epsilon$ and using the fact that $t$ is a distance function in the original $g$ metric, one sees that $s_\epsilon=0$.
	Norming this expression now gives
	\begin{eqnarray}
		\begin{aligned}
			|\Ric{}_\epsilon|^2_{g_\epsilon}
			&\;=\;
			\frac{192\epsilon^4}{t^8(\epsilon^2+t^{-2})^8}
			\;=\;\frac{192\epsilon^4}{(\rho^2+4\epsilon^2)^4}
		\end{aligned}
	\end{eqnarray}
	and so $|\Ric_\epsilon|=O(\rho^{-4})$, as claimed.
\end{proof}


\begin{wrapfigure}{r}{0.55\textwidth}
	\includegraphics[scale=0.7]{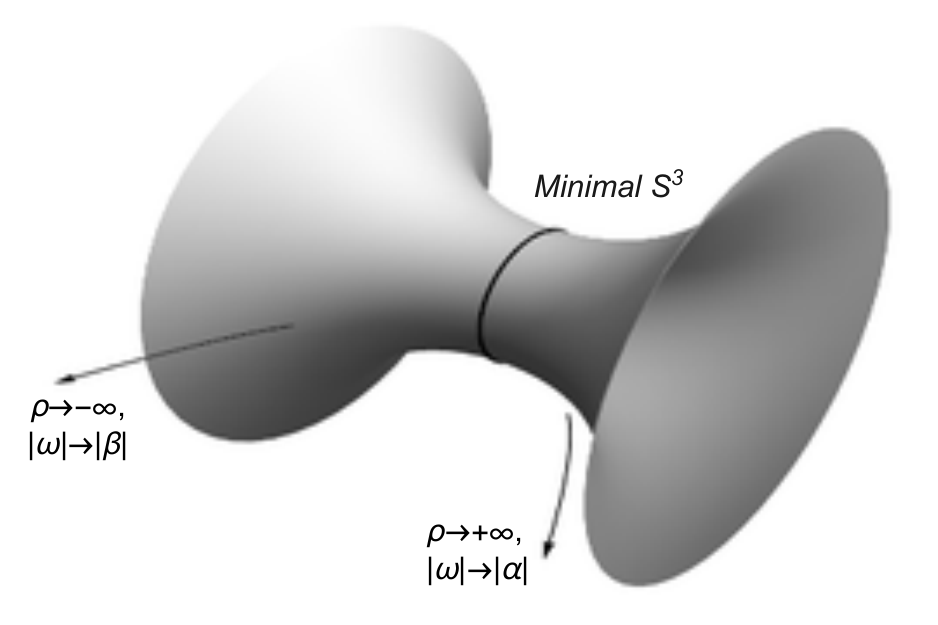}
	\caption{2-ended AE manifold with metric (\ref{Eqn2EndedMetric}) and closed 2-form (\ref{EqnAKForm}).}
\end{wrapfigure}
Now we construct the closed form $\omega\in\bigwedge^+$ that is asymptotically K\"ahler at both ends.
Let $\eta_2$, $\eta_{-2}$ be the divergence-free covector fields on $\mathbb{S}^3$ with $*d\eta_{\pm2}=\pm2\eta_{\pm2}$.
From Theorem \ref{ThmEigenvaluesRestriction} we know $|\eta_2|$ and $|\eta_{-2}|$ are pointwise constant on $\mathbb{S}^3$.
Choose the standard normalizaation $|\eta_2|\equiv1$ and $|\eta_{-2}|\equiv1$ on $\mathbb{S}^3$---we remark that the inner product $\left<\eta_2,\eta_{-2}\right>$ is not constant on $\mathbb{S}^3$.
Referencing the correspondance between time-varying 1-forms on $\mathbb{S}^3$ and 2-forms on $\mathbb{R}^4$, consider the 2-forms $\omega_{2}\in\bigwedge^+$ and $t^{-4}\omega_{-2}\in\bigwedge^+$, which are both closed by Proposition \ref{PropCorrespondance}.
Then let $\omega$ be the closed 2-form
\begin{eqnarray}
	\omega\;\triangleq\;\alpha\epsilon^4\omega_2\,+\,\beta{}t^{-4}\omega_{-2}\;\in\;\bigwedge{}^+ \label{EqnAKForm}
\end{eqnarray}
for any given $\alpha,\beta\in\mathbb{R}$.
Computing the norms in the $g_\epsilon$-metric, we have
\begin{eqnarray}
	\begin{aligned}
		|\omega|_{g_\epsilon}^2&\;=\;
		\alpha^2\epsilon^8|\omega_2|_{g_\epsilon}^2
		+2\alpha\beta{}\epsilon^4t^{-4}\left<\omega_2,\omega_{-2}\right>_{g_\epsilon}
		+\beta^2t^{-8}|\omega_{-2}|_{g_\epsilon}^2 \\
		&\;=\;\alpha^2\epsilon^8(\epsilon^2+t^{-2})^{-4} \\
		&\quad
		\,+\,2\alpha\beta{}\epsilon^4t^{-4}(\epsilon^2+t^{-2})^{-4}\left<\eta_2,\eta_{-2}\right>_{\mathbb{S}^3} \\
		&\quad
		\,+\,\beta^2t^{-8}(\epsilon^2+t^{-2})^{-4}.
	\end{aligned} \label{EqnNormOmega}
\end{eqnarray}
From this we compute the asymptotic behavior of $|\omega|^2_{g_\epsilon}$ toward the two ends of the manifold, $\rho\rightarrow\infty$ and $\rho\rightarrow-\infty$.
Using $t=\rho+\sqrt{\rho^2+4\epsilon^2}$ we estimate
\begin{eqnarray}
	\begin{aligned}
		&\alpha^2\epsilon^8(\epsilon^2+t^{-2})^2
		\;=\;
		\begin{cases}
			\alpha^2\left(1-4\epsilon^2\rho^{-2}\,+\,O(\rho^{-4})\right), \quad & \rho\rightarrow+\infty \\
			\alpha^2\left(\epsilon^8\rho^{-8}\,+\,O(\rho^{-10})\right), \quad & \rho\rightarrow-\infty
		\end{cases} \\
		&2\alpha\beta\epsilon^4{}t^{-4}(\epsilon^2+t^{-2})^{-4}
		\;=\;
		\begin{cases}
			2\alpha\beta\left(\epsilon^4\rho^{-4}\,+\,O(\rho^{-6})\right), \quad & \rho\rightarrow+\infty \\
			2\alpha\beta\left(\epsilon^4\rho^{-4}\,+\,O(\rho^{-6})\right), \quad & \rho\rightarrow-\infty
		\end{cases} \\
		&\beta^2t^{-8}(\epsilon^2+t^{-2})^{-4}
		\;=\;
		\begin{cases}
			\beta^2\left(\epsilon^8\rho^{-8}\,+\,O(\rho^{-10})\right), \quad & \rho\rightarrow+\infty \\
			\beta^2\left(1-4\epsilon^2\rho^{-2}\,+\,O(\rho^{-4})\right), \quad & \rho\rightarrow-\infty
		\end{cases}
	\end{aligned} \label{EqnsAsymptoticsVarious}
\end{eqnarray}
and using $|\left<\eta_2,\eta_{-2}\right>_{\mathbb{S}^3}|\le|\eta_2||\eta_{-2}|=1$ we therefore have
\begin{eqnarray}
	|\omega|_{g_\epsilon}^2
	\;=\;
	\begin{cases}
		\alpha^2+O(\rho^{-2}), \quad & \rho\rightarrow+\infty \\
		\beta^2+O(\rho^{-2}), \quad & \rho\rightarrow-\infty.
	\end{cases} \label{EqnTwoEndedOmegaAsymptotics}
\end{eqnarray}

Assuming $\alpha\ne0$ then $\omega$ is asymptotically K\"ahler as $\rho\rightarrow\infty$ and if $\beta\ne0$ then $\omega$ is asymptotically K\"ahler as $\rho\rightarrow-\infty$.
If only one of $\alpha,\beta$ is zero, then $\omega$ is asymptotically K\"ahler along one end of the manifold and asymptotically zero along the other.

{\bf Remark}.
In Theorem \ref{PropDecayGapOfForms} we proved that if $\omega$ is asymptotically 0, then $|\omega|=O(\rho^{-4})$---decay rates of $\rho^{-1}$, $\rho^{-2}$, $\rho^{-3}$ are forbidden.
The expressions in (\ref{EqnsAsymptoticsVarious}) provide a concrete example of this.
If $\beta=0$ say, then as $\rho\rightarrow-\infty$ the form $\omega$ is asymptotically 0, and we see from (\ref{EqnsAsymptoticsVarious}) that indeed $|\omega|=|\alpha|\epsilon^{-4}\rho^{-4}+O(\rho^{-5})$ as $\rho\rightarrow-\infty$.

\vspace{0.1in}

Finally, to prove that $|\nabla\omega|$ cannot be controlled in the bubble by any multiple of $L^2(|\nabla\omega|)$ we prove that as $\epsilon\searrow0$---and the 2-ended bubble converges to the 1-point union of two copies of $\mathbb{R}^4$---then $L^\infty(|\nabla\omega|)\nearrow\infty$ even while $L^2(|\nabla\omega|)\searrow0$.

From expression (\ref{EqnNormOmega}), certainly $\frac{\partial}{\partial{\rho}}|\omega|^2$ becomes infinite.

Now we compute $L^2(|\nabla\omega|)$ as follows.
From (\ref{EqnMainBochner}) we know $\triangle\omega=0$, so choosing some large $A$ integration by parts gives
\begin{eqnarray}
	\begin{aligned}
		\int_{-A<\rho<A}|\nabla\omega|^2
		&=
		\frac12\int_{\rho=A}\frac{\partial}{\partial\rho}|\omega|^2\,d\sigma
		-\frac12\int_{\rho=-A}\frac{\partial}{\partial\rho}|\omega|^2\,d\sigma.
	\end{aligned} \label{EqnL2IntByParts}
\end{eqnarray}
We have that $\int_{\mathbb{S}^3}\left<\eta_2,\eta_{-2}\right>=0$, and using $t=(\rho+\sqrt{\rho^2+4\epsilon^2})/2\epsilon^2$ (and perhaps some computer assistance to avoid tedious computations by hand), plugging the expression (\ref{EqnNormOmega}) into (\ref{EqnL2IntByParts}) gives
\begin{eqnarray}
	\begin{aligned}
		\int_{-A<\rho<A}|\nabla\omega|^2
		&\;=\;|\mathbb{S}^3|\cdot8\alpha^2\epsilon^2\left(1+O(A^{-2})\right).
	\end{aligned}
\end{eqnarray}
Now taking the limit $A\rightarrow\infty$, we obtain $L^2(|\nabla\omega|)=18\epsilon^2\pi^2$.
We have proved the following proposition.
\begin{proposition}
	As $\epsilon\searrow0$ the manifold $(\mathbb{R}^4\setminus\{pt\},g_\epsilon)$ converges, in the Gromov-Hausdorff topology, to the one-point union of two copies of $\mathbb{R}^4$.
	As $\epsilon\searrow0$ the form $\omega$ from (\ref{EqnAKForm}) converges to a covariant-constant form of magnitude $|\alpha|$, $|\beta|$, respectively, on each copy.
	
	Further, $L^2(|\nabla\omega|)\searrow0$ while $\sup|\nabla\omega|\nearrow\infty$ as $\epsilon\searrow0$.
\end{proposition}

\subsection{Topologically non-trivial bubbles}

In this concluding section, we present an example of a multi-ended bubble that is not conformally Euclidean and indeed is topologically non-trivial, but still has a bounded, closed 2-form $\omega\in\bigwedge^+$.

Let $(N^4,g)$ be the Burns metric on $O(-1)$ over $\mathbb{P}^1$, the Eguchi-Hansen metric on $O(-2)$, or any of the LeBrun metrics \cite{L3} on $O(-k)$, $k>2$.
These metrics are all K\"ahler and half-fonformally flat.

Then $g$ has K\"ahler form $\omega\in\bigwedge^+$, and $g$ is ALE with group $\Gamma$ of order $|\Gamma|=k$.
In the Burns case, $g$ is AE.
Each manifold has second betti number $b^2=1$, and is contractible to a 2-sphere that represents this cohomology, colloquially called its ``bolt.''

Pick any finite set of points $p_1,\dots,p_l\in{}N^4$, and for each $p_i$ let $G_i:N^4\setminus\{p_i\}\rightarrow\mathbb{R}$ be the positive Green's function at that point, meaning $\triangle{}G_i=-\delta_{p_i}$ and $\inf_{N^4}G_i=0$.
Let $\hat{g}$ be the conformally related metric
\begin{eqnarray}
	\hat{g}\;=\;\epsilon^2\left(1+\sum_{i=1}^lG_i\right)^2g.
\end{eqnarray}
Because $1+\sum_{i=1}^lG_i$ is harmonic, $\hat{g}$ is scalar-flat.
Certainly $\hat{g}$ is half-conformally flat, by the conformal invariance of $W^+$ and $W^-$.
By reasoning silimar to that in \S\ref{SubSec2EndedAK}, the metric around each point $p_i$ becomes asymptotically Euclidean (this is a well-known fact, and we omit the tedious but straighforward computations).
As $\epsilon\rightarrow0$ the manifold $(N^4\setminus\{p_1,\dots,p_l\},\hat{g})$ converges to the 1-point union of $l$ many copies of $\mathbb{R}^4$ and one copy of $\mathbb{R}^4/\Gamma$, where $\Gamma$ is the cyclic group of order $k$.

The closed 2-form $\omega'=\epsilon^2\omega$ on $N^4\setminus\{p_1,\dots,p_l\}$ is not K\"ahler or covariant-constant in the $\hat{g}$-metric.
In fact its norm is
\begin{eqnarray}
	|\omega'|_{\hat{g}}\;=\;\frac{1}{\left(1+\sum_{i=1}^lG_i\right)^2}.
\end{eqnarray}
Because $G_i$ grows like $r^{-2}$ near $p_i$ we have that $|\omega'|_{\hat{g}}=O(\rho^{-4})$ along the end created near $p_i$, where $\rho$ is the distance in the $\hat{g}$-metric to some fixed point.
This is an ``optimal'' decay rate, according to Theorem \ref{PropDecayGapOfForms}.

Thus $\omega'$ is asymptotically zero along all ends created with the Green's functions.
On the original manifold end, we see that $\omega'$ is asymptotically K\"ahler.
This is easily seen from the fact that $|\omega'|_{\hat{g}}=O(1)$ along this end, and then using Theorem \ref{PropDecayGapOfForms}.



\begin{thebibliography}{9}
	
	
	\bibitem{BKN}{S. Bando, A. Kasue and H. Nakajima}, \emph{On a construction of coordinates at infinity on manifolds with fast curvature decay and maximal volume growth},  {Inventiones Mathematicae} {\bf 97} (1989) No. 2 {313--349}
	
	\bibitem{Besse}{A. Besse}. ``Einstein manifolds'' Springer Science \& Business Media, 2007.
	
	\bibitem{Bra}{T. Branson}, \emph{Kato constants in Riemannian geometry} Mathematical Research Letters 7, no. 3 (2000): 245-261.
	
	\bibitem{Brill} {D. Brill}, \emph{Euclidean Maxwell-Einstein theory.} Topics on Quantum Gravity and Beyond: Essay in Honor of Louis Witten on His Retirement. F. Mansouri and JJ Scanio, eds.(World Scientific: Singapore) (1993).
	
	\bibitem{Bry}{R. Bryant}, "Nine lectures on exterior differential systems." Informal notes for a series of lectures delivered (1999): 12-23.
	
	\bibitem{CGH}{D. Calderbank, P. Gauduchon, and M. Herzlich}, \emph{Refined Kato inequalities and conformal weights in Riemannian geometry.} Journal of Functional Analysis 173, no. 1 (2000): 214-255.
	
	\bibitem{F} {G. Folland}, \emph{Harmonic analysis of the de Rham complex on the sphere}, J. reine angew. Math 398 (1989): 130-143.
	
	\bibitem{Gao1}{Z. Gao}, \emph{Convergence of Riemannian manifolds; Ricci and $L^{n/2}$-curvature pinching}. Journal of Differential Geometry 32, no. 2 (1990): 349-381.
	
	\bibitem{Gao2}{Z. Gao}, \emph{$L^{n/2}$-curvature pinching}. Journal of Differential Geometry 32, no. 3 (1990): 713-774.
	
	\bibitem{GT} {D. Gilbarg and N. Trudinger}, Elliptic partial differential equations of second order. springer, 2015.
	
	\bibitem{Gro}{M. Gromov}, \emph{Curvature, diameter and Betti numbers}. Commentarii Mathematici Helvetici 56, no. 1 (1981): 179-195.
	
	\bibitem{Her1}{J. Heras}, \emph{Electromagnetism in Euclidean four space: A discussion between God and the Devil.} American Journal of Physics 62, no. 10 (1994): 914-916.
	
	\bibitem{Her2}{J. Heras} \emph{The kirchhoff gauge}. Annals of Physics 321, no. 5 (2006): 1265-1273.
	
	\bibitem{LeB}{C. LeBrun}, \emph{On the topology of self-dual 4-manifolds}, {Proceedings of the American Mathematical Society}, {\bf 98} (1986) no. 4 {637--640}
	
	\bibitem{L3}{C. Lebrun}, \emph{Counterexamples to the generalized positive action conjecture.} \emph{Communications in Mathematical Physics} 118 (1988): 591--596
	
	\bibitem{L4}{C. Lebrun}, \emph{Curvature functionals, optimal metrics, and the differential topology of 4-manifolds}. In Different faces of geometry (pp. 199-256). Springer, Boston, MA.
	\emph{Communications in Mathematical Physics} 118 (1988): 591--596
	
	\bibitem{LTZ}{L. Lindblom, N. Taylor, and F. Zhang}, \emph{Scalar, vector and tensor harmonics on the three-sphere.} General Relativity and Gravitation 49, no. 11 (2017): 139.
	
	\bibitem{Mink} {H. Minkowski}, \emph{Die Grundgleichungen f\"ur die elektromagnetischen Vorg\"ange in bewegten K\"orpern.} Nachrichten von der Gesellschaft der Wissenschaften zu Göttingen, Mathematisch-Physikalische Klasse 1908 (1908): 53-111.
	
	\bibitem{PSW}{P. Peterson, S. Shteingold, and G. Wei}, \emph{Comparison geometry with integral curvature bounds.} Geometric \& Functional Analysis GAFA 7, no. 6 (1997): 1011-1030.
	
	\bibitem{Sand}{V. Sandberg}, \emph{Tensor sherical harmonics on $\mathbb{S}^2$ and $\mathbb{S}^3$ as eigenvalue problems}. Journal of Mathematical Physics Vol 12, no. 12 (1978): 2441--2446
	
	\bibitem{Sea}{W. Seaman}, \emph{Harmonic two-forms in four dimensions}. Proceedings of the American Mathematical Society (1991): 545-548.
	
	\bibitem{Schw}{J. Schwinger} \emph{Euclidean quantum electrodynamics.} Physical Review 115, no. 3 (1959): 721.
	
	\bibitem{T}{C. Taubes}, \emph{The extistence of anti-self dual conformal structures}, {Journal of Differential Geometry} {\bf 36} (1992) {163--253}
		
	\bibitem{TV1}{G. Tian and J. Viaclovsky}, \emph{Bach-flat symptotically locally Euclidean metrics}, {Invetiones Mathematicae} {\bf 160} (2005) No. 2 {357--415}
	
	\bibitem{TV2}{G. Tian and J. Viaclovsky}, \emph{Moduli spaces of critical Riemannian metrics in dimensiona four}, {Advances in Mathematics} {\bf 196} (2005) {346--372}
	
	\bibitem{TV3}{G. Tian and J. Viaclovsky}, \emph{Volume growth, curvature decay, and critical metrics}, {arXiv preprint math/0612491} (2006)
	
	\bibitem{Web1}{B. Weber}, \emph{Two classical flows and the topology of 4-manifolds}, {To appear.} (2017)
	
	\bibitem{Y}{D. Yang}, \emph{Convergence of Riemannian manifolds with integral bounds on curvature}. I." In Annales scientifiques de l'Ecole normale supérieure, vol. 25, no. 1, pp. 77-105. 1992.
	
	\bibitem{Zamp1}{E. Zampino}, \emph{A brief study on the transformation of Maxwell equations in Euclidean four‐space}, Journal of mathematical physics 27, no. 5 (1986): 1315-1318.
	
	\bibitem{Zamp2}{E. Zampino} "Can A" Hyperspace" Really Exist?." NASA 19990023249 (1999).

\end{thebibliography}
\end{document}